\pdfoutput=1
\documentclass{amsart}
\usepackage[utf8]{inputenc}

\usepackage{geometry}
\newenvironment{Comment}[2]{\noindent\color{#1}{\texttt #2: }}{}

\usepackage{amsmath}
\usepackage{amsthm}
\usepackage{mathrsfs}
\usepackage{amsfonts}
\usepackage{ amssymb }
\usepackage{amsmath,amscd}
\usepackage{array}
\usepackage{amssymb}
\usepackage[all, cmtip]{xy}
\usepackage{tikz}
\usepackage{enumitem}
\usepackage{tensor}
\usetikzlibrary{arrows}
\usetikzlibrary{cd}
\usepackage[backend=biber,
style=alphabetic,
isbn=false,
url=false,
sorting=nyt]{biblatex}
\usepackage{bm}
\newtheorem{theorem}{Theorem}[subsection]
\newtheorem{lemma}[theorem]{Lemma}

\newtheorem{proposition}[theorem]{Proposition}
\newtheorem{definition}[theorem]{Definition}
\newtheorem{corollary}[theorem]{Corollary}
\newtheorem{remark}[theorem]{Remark}

\newcommand{\CC}{\mathbb{C}}
\newcommand{\QQ}{\mathbb{Q}}

\newcommand{\PP}{\mathbb{P}}
\newcommand{\EE}{\mathbb{E}}

\newcommand{\cF}{\mathcal{F}}
\renewcommand{\AA}{\mathbb{A}}
\newcommand{\TT}{\mathbb{T}}
\newcommand{\LL}{\mathbb{L}}
\newcommand{\T}{\mathscr{T}}
\newcommand{\ZZ}{\mathbb{Z}}

\newcommand{\bS}{\mathbf{S}}
\newcommand{\OO}{\mathscr{O}}

\newcommand{\Mbar}[2]{\mathcal{\overline{M}}_{#1, #2}}
\newcommand{\cZ}{\mathcal{Z}}
\newcommand{\bu}{\mathbf{u}}
\newcommand{\bx}{\mathbf{x}}

\newcommand{\g}{\mathfrak{g}}
\renewcommand{\P}{\mathscr{P}}
\renewcommand{\L}{\mathscr{L}}
\newcommand{\D}{\mathscr{D}}

\newcommand{\rpic}{r_{\mathrm{Pic}}}
\newcommand{\rchi}{r_\chi}
\newcommand{\bdtilde}{{r_T}}
\newcommand{\bd}{r_G}

\newcommand{\bxi}{\pmb{\xi}}

\newcommand{\sslash}{\mathord{/\mkern-6mu/}}

\renewcommand{\Im}{\mathrm{im}}

\newcommand{\fix}{\mathrm{fix}}
\newcommand{\mov}{\mathrm{mov}}
\newcommand{\vir}{\mathrm{vir}}
\newcommand{\Bun}{\mathfrak{B}un}
\newcommand{\specialP}{\mathbb{P}^1}

\newcommand{\csigma}{\mathcal{S}}
\newcommand{\twist}{a}
\newcommand{\bc}{\mathbf{c}}
\newcommand{\fU}{\mathfrak{U}}

\newcommand{\evhat}{\hat{ev}}
\newcommand{\II}{\mathbb{I}}
\newcommand{\bt}{\mathbf{t}}

\newcommand{\Sec}[3]{\underline{\mathrm{Sec}}_{#1}(#2/#3)}

\newtheorem{example}{Example}

\DeclareMathOperator{\Aut}{Aut}

\DeclareMathOperator{\Hom}{Hom}

\DeclareMathOperator{\Pic}{Pic}
\DeclareMathOperator{\Spec}{Spec}
\DeclareMathOperator{\Bunset}{Bun}

\begin{document}

\title{The Abelian-Nonabelian Correspondence for $I$-functions}
\author{Rachel Webb}

\address[R. Webb]{Department of Mathematics\\
University of California, Berkeley\\
Berkeley, CA 94720-3840\\
U.S.A.}
\email{rwebb@berkeley.edu}

\date{\today}

\begin{abstract}
We prove the abelian-nonabelian correspondence for quasimap $I$-functions. That is, if $Z$ is an affine l.c.i. variety with an action by a complex reductive group $G$, we prove an explicit formula relating the quasimap $I$-functions of the GIT quotients $Z\sslash_{\theta} G$ and $Z\sslash_{\theta} T$ where $T$ is a maximal torus of $G$. We apply the formula to compute the $J$-functions of some Grassmannian bundles on Grassmannian varieties and Calabi-Yau hypersurfaces in them.
\end{abstract}

\keywords{quasimaps, I-function, abelianization, moduli of maps from the projective line}

\subjclass[2010]{14L30, 14N35, 14D23}

\maketitle

\setcounter{tocdepth}{1}
\tableofcontents

\section{Introduction}
Let $Z$ be an affine l.c.i. variety and let $G$ be a connected complex reductive algebraic group acting on $Z$ with maximal torus $T$. A character $\theta$ of $G$ defines a linearization of the trivial bundle on $Z$. From this data, we get two GIT quotients with a rational map between them: $Z\sslash T \dashrightarrow Z \sslash G$. The \textit{abelian-nonabelian correspondence} is a conjectured relationship \cite[Conj~3.7.1]{frob-ab-nonab} between the genus-zero Gromov-Witten invariants of $Z\sslash G$ and those of $Z \sslash T$. This paper proves a correspondence of their small quasimap $I$-functions, and of the big $I$-functions when $Z$ is affine space. In certain cases, our $I$-function correspondence implies \cite[Conj~3.7.1]{frob-ab-nonab}.


\subsection{Statement of the main result}\label{sec:statement}
Let $Z$, $G$, $T$, and $\theta$ be as stated. We work in the setting of \cite[Section~2.1]{wcgis0}.
Let $Z^s(G)$ and $Z^{ss}(G)$ denote the $\theta$-stable and semistable points in $Z$. We assume that $Z^s(G) = Z^{ss}(G)$ is smooth and not empty and that $G$ acts on $Z^s(G)$ freely. We also assume that each of these statements holds with $T$ in place of $G$. Hence $V\sslash_{\theta}G$ and $V\sslash_{\theta}T$ are smooth varieties. We fix the character $\theta$ for all of this paper, and we will generally write $Z\sslash G := Z\sslash_\theta G$ and $Z\sslash T := Z\sslash_\theta T$.

The small quasimap $I$-function of $Z\sslash G$, defined in \cite{stable_qmaps}, has the form
\begin{equation}\label{eq:Ifunc_shape}
I^{Z\sslash G}(z) = 1+ \sum_{\beta\neq 0}q^\beta I^{Z\sslash G}_\beta(z) 
\end{equation}
where $\beta$ is in $\Hom(\Pic^G(Z), \ZZ)$ with $\Pic^G(Z)$ the group of $G$-equivariant line bundles on $Z$, $q^{\beta}$ is a formal variable, and the coefficients $I^{Z\sslash G}_{\beta}(z)$ are formal series in $z$ and $z^{-1}$ with coefficients in $H^*(Z\sslash G)$.

To state the main theorem, we make two observations. First, the rational map $Z\sslash T \dashrightarrow Z\sslash G$ may be stated more precisely via the diagram
\begin{equation}\label{eq:key_diagram}
\begin{tikzcd}
Z^s(G)/T \arrow[hook, r, "j"] \arrow[d, "g"] & Z^s(T)/T\\
Z^s(G)/G &
\end{tikzcd}
\end{equation}
Second, there is an inclusion 
\begin{equation}\label{eq:inclusion}
\chi(G) \rightarrow \Pic^G(Z) 
\end{equation}
sending the character $\xi$ to \begin{equation}\label{eq:defL}
\L_{\xi}:= Z \times \CC_{\xi}\end{equation}
where $\CC_{\xi}$ is the representation with character $\xi$. The line bundle $\L_\xi$ descends to a line bundle on $Z\sslash G$ which we also denote by $L_\xi$.

\begin{theorem}\label{thm:main}
The small $I$-functions of $Z\sslash G$ and $Z\sslash T$ satisfy
\begin{equation}\label{eq:main}
g^* I^{Z\sslash G}_{\beta}(z) = j^*\left[ \sum_{\tilde \beta \rightarrow \beta} \left( \prod_{\rho} \frac{\prod_{k=-\infty}^{\tilde\beta(\rho)}(c_1(\L_{\rho}) + kz)}{\prod_{k=-\infty}^0 (c_1(\L_{\rho}) + kz)}\right)I^{Z\sslash T}_{\tilde \beta}(z)\right],
\end{equation}
where the sum is over all $\tilde \beta$ mapping to $\beta$ under the natural map $\Hom(\Pic^T(Z), \ZZ) \rightarrow \Hom(\Pic^G(Z), \ZZ)$ and the product is over all roots $\rho$ of $G$.
\end{theorem}
Since $g^*$ is injective, the equality \eqref{eq:main} completely determines $I^{Z\sslash G}$; we explain in Section \ref{sec:interpret} how to make sense of this sentence. When $Z$ is a vector space, one obtains a closed formula for $I_\beta^{Z\sslash G}$ by combining \eqref{eq:main} with the formula for $I^{Z\sslash T}_{\tilde \beta}$ in \cite[Thm~5.4]{orb-qmaps} (which agrees with the formula in \cite{givental}).

Previous to this work, Theorem \ref{thm:main} was known for $Z \sslash G$ equal to a flag variety or the Hilbert scheme of $n$ points in $\CC^2$ \cite{twoproofs} \cite{ab-nonab} \cite{hilbert}. Since the posting of this paper, a proof of this result for quiver flag varieties has also appeared \cite{elana}. Finally, a result similar to Theorem \ref{thm:main} was proved in the language of stable gauged maps in \cite[Thm~1.4,~Thm~3.9]{gonzalez}. Gauged maps apply in a different context than quasimaps: while quasimap theory requires $Z$ to be an affine l.c.i. variety, the theorey of gauged maps mandates that $Z$ be smooth projective. It is not clear how the theories of gauged maps and quasimaps are related \cite[Rmk~4.4]{orb-qmaps}.

We remark that Theorem \ref{thm:main} completes a provisional result in \cite{inoue}. The bulk of the proof of Theorem \ref{thm:main} is a careful analysis of certain moduli spaces of maps from $\PP^1$ to $[Z/G]$ and $[Z/T]$; the geometry of these moduli spaces is summarized in Proposition \ref{prop:diagram}. This geometry may be useful in other contexts. For instance, it is used in \cite{yaoxiong} to compute quasimap $I$-functions in $K$-theory.

\subsection{Extensions and applications}
This paper also proves the natural extension of Theorem \ref{thm:main} to the equivariant and twisted theories (Corollaries \eqref{cor:equivariant} and \eqref{cor:twisted}). We also use the theory developed in \cite{bigI} to write down a big $I$-function for $Z\sslash_\theta G$ when $Z$ is a vector space (Corollary \ref{cor:bigI}). This recovers, for example, an explicit big $I$-function for the Grassmannian $Gr(k, n)$.

With the reconstruction result in \cite{frob-ab-nonab} and the mirror result in \cite{wcgis0}, the equivariant version of Theorem \ref{thm:main} implies the full abelian-nonabelian correspondence for projective Fano quotients $Z\sslash G$ with ``nice'' torus actions (Corollary \ref{cor:fm-abelianization}). We use this result to explicitly compute an equivariant twisted small $J$-function for a Grassmannian bundles on a Grassmannian variety (Theorem \ref{thm:example}).

\subsection{Conventions and notation}\label{sec:notation}
We work over $\CC$. A variety is an integral separated finite type scheme. Fix an affine variety $Z$ with a left action by $G$ that is a complex reductive algebraic group over $\CC$ and fix $T\subset G$ a maximal torus. Let $N_G(T)$ be the normalizer of $T$ in $G$, and let $W = N_G(T)/T$ denote the Weyl group. The characters of $G$ are $\chi(G)$.

Fix a character $\theta \in \chi(G)$. Let $Z^s(G)$ and $Z^{ss}(G)$ denote the $\theta$-stable and semistable loci as defined in \cite[Section~2]{king}. We assume that $Z^s(G) = Z^{ss}(G)$ is smooth and not empty and that $G$ acts on $Z^s(G)$ freely. We also assume that each of these statements holds with $T$ in place of $G$. Hence $V\sslash_{\theta}G$ and $V\sslash_{\theta}T$ are smooth varieties, projective over $\bS_G := \Spec(H^0(Z, \OO_Z)^G)$ and $\bS_T :=\Spec(H^0(Z, \OO_Z)^T),$ respectively.

\subsection{Organization of the paper}
In Sections 2 and 3 we review standard facts about abelian/nonabelian correspondences and quasimaps, respectively, writing out many well-known formulae explicitly so that we can reference them in computations later. Section 4 explains the geometry behind Theorem \ref{thm:main} and Section 5 completes the proof of this theorem. Section 6 proves various extensions and applications of the main result.

\subsection{Acknowledgements} Thanks to Robert Silversmith for explaining quasimaps to me in a concrete way, and to Yongbin Ruan for helpful discussions. Thank you to Jonathan Wise, Damiano Fulghesu, Matthew Satriano, and Martin Olsson for teaching me enough algebraic geometry so I could write this paper rigorously. An anonymous referee provided many helpful comments and corrections. This project was partially supported by NSF RTG grant 1045119 and an NSF Postdoctoral Research Fellowship, award number 200213.
\section{Some first abelian/nonabelian correspondences}
\subsection{Preliminaries on principal bundles}
A principal $G$-bundle on a scheme $C$ is a scheme $\pi: \P\rightarrow C$ with $\pi$ faithfully flat and locally finitely presented, together with an action $\mu: G \times \P \rightarrow \P$ leaving $\pi$ invariant such that the map
\[
G \times \P \xrightarrow{(\mu, pr_2)} \P \times_C\P
\]
is an isomorphism. With our assumptions, a principal $G$-bundle is locally trivial in the \'etale topology (see eg \cite[Rmk~4.5.7]{olsson-book}).

If $\P \rightarrow C$ is a principal $G$-bundle on a scheme $C$ and $Z$ is an affine variety with a left $G$-action, then we define the \textit{mixing space} to be the quotient
\begin{equation}\label{eq:principal1}
\P \times_G Z = (\P\times Z)/G \quad \quad \text{where}\;g\cdot(p, z) = (gp, gz) \;\text{for}\; g \in G, \;(p,z) \in \P\times Z.
\end{equation}
This space is an \'etale-locally trivial fibration on $C$ with fiber $Z$. A priori it is an algebraic space, but for example if $C$ is finite type then it is known to be a scheme (using affineness of $Z$, see eg \cite[Sec~3]{brion11}).

Our convention in \eqref{eq:principal1} is necessary because we want the stack quotient $Z \rightarrow [Z/G]$ to be represented by principal $G$-bundles. However, it has the unfortunate consequence that the transition functions of $\P$ and $\P\times_G Z$ are inverse to each other. Hence, if $T \subset G$ is a subgroup and $\T\rightarrow C$ is a principal $T$-bundle, we define the \textit{associated G-bundle} to be
\begin{equation}\label{eq:principal2}
G \times_T \T = (G \times \T)/T \quad \quad \text{where}\;t\cdot(g, s) = (gt^{-1}, ts)\;\text{for}\;t \in T, \;(g,s) \in G \times \T.
\end{equation}
The quotient $G \times_T \T$ is a (left) principal $G$-bundle with the same transition function as $\T$.

\subsection{Action of the Weyl group}\label{sec:weylaction}
Heuristically, an abelian/nonabelian correspondence relates data of $G$ to data of $T$ and the Weyl group $W$. In this section we explain the action of $W$ on several objects of interest. 

The Weyl group acts on $T$, characters of $T$, dual characters and cocharacters of $T$ in the usual ways; i.e., for  $w \in N_G(T)$ and $t \in T$, $\xi \in \chi(T)$ a character, $\tilde \alpha \in \Hom(\chi(T), \ZZ)$ a dual character, and $\tau: \CC^* \rightarrow T$ a cocharacter, we have
\begin{equation}\label{eq:dual_actions}
\begin{aligned}
w \cdot t & = wtw^{-1} &
w\cdot \xi(t) &= \xi(w^{-1}\cdot t) \\
w \cdot \tilde \alpha (\xi) &= \tilde \alpha(w^{-1}\cdot \xi) &
(w \cdot \tau)(t) &= w\cdot(\tau(t)) .
\end{aligned}
\end{equation}

Let $Z$ be a $G$-scheme and assume that the quotient $[Z/T]$ is represented by a smooth scheme $Z/T$ (e.g., $Z/T$ is one of the GIT targets considered in this paper). Then the action of the group scheme $W$ on $Z$ descends to $Z/T$ as follows. For $w \in N_G(T)$ and $t \in T$, $z \in Z$ we compute
\begin{equation}\label{eq:twisted}
w(tz) = (wtw^{-1})(wz).
\end{equation}
This shows that the action of $N_G(T)$ descends to $Z/T$, and clearly the action of $T \subset N_G(T)$ is trivial.
Because of the computation \eqref{eq:twisted}, we say that the map $w: Z \rightarrow Z$ is \textit{twisted-equivariant} for the homomorphism $\twist_w: T \rightarrow T$ defined by $\twist_w(t) = wtw^{-1}$; that is, 
\[
w(tz) = \twist_w(t)w(z).
\]
This twisted equivariance manifests itself in the Weyl actions that we describe below.

\subsubsection{On maps to $[Z/T]$}\label{sec:wonmaps}
Even when $[Z/T]$ is not representable it has an action by the group scheme $W$ by \cite[Rmk~2.4]{romagny}, but this action is more subtle. It remembers the fact that as computed in \eqref{eq:twisted}, the morphism $w:Z \rightarrow Z$ is not a morphism of $T$-schemes. We now describe it explicitly.

A morphism $S\rightarrow[Z/T]$ from a scheme $S$ is given by a principal $T$-bundle $\pi:\T\rightarrow S$ and an equivariant map $\T \rightarrow Z$. For $w$ an $S$-point of $N_G(T)$, the morphism $w \cdot (S\rightarrow [Z/T])$ is given by the same data, but with the action on $\T$ twisted by $a_w^{-1}$ (the twisted action is written out in \eqref{eq:weylonbunT} below). This follows from the following commuting diagram (where for a moment we forget the $T$-actions).
\begin{equation}\label{eq:weyl1}
\begin{tikzcd}
\T \arrow[r] \arrow[d, "\pi"] & Z \arrow[r, "w"] \arrow[d] & Z \arrow[d]  \\
S \arrow[r] & {[Z/T]}  \arrow[r, dashrightarrow] & {[Z/T]} 
\end{tikzcd}
\end{equation}
Here, $w$ denotes the automorphism of $Z(\T)$ defined by $\pi^*w \in N_G(\T)$, and the dashed arrow is the automorphism of $[Z/T](S)$ defined by $w$ as in \cite[Rmk~2.4]{romagny}.\footnote{In the language of \cite{romagny}, the composition $Z \xrightarrow{w} Z \rightarrow [Z/T]$ is a \textit{non-strict} morphism of $T$-stacks (the 2-morphism in \cite[(2)]{romagny} is given by the twist $a$). Nevertheless by the universal property of $[Z/T]$ there is an induced morphism $[Z/T] \rightarrow [Z/T]$ making the right square in \eqref{eq:weyl1} 2-commute, and this defines the Weyl action on $[Z/T]$.} With the original $T$-action on $Z$, the map $w$ is not equivariant, but when we twist the $T$-actions on $\T$ and the middle $Z$ by $a_w^{-1}$, both of the top maps in \eqref{eq:weyl1} are equivariant.

For example, when $Z$ is a point we get an action on principal bundles: define $w\T$ to be the $T$-bundle with the same underlying space as $\T$ and with $T$-action $\cdot_w$ given by
\begin{equation}\label{eq:weylonbunT}
t \cdot_w x = (w^{-1}tw)\cdot x \quad \quad \text{for}\;t\in T, \; x \in \T.
\end{equation}
Here $\cdot$ denotes the usual action of $T$ on $\T$. The identity map $\T \rightarrow w\T$ is an $\twist_w$-equivariant isomorphism.

For use when $Z$ is nontrivial, we also describe the action on $[Z/T]$ in terms of sections of fiber bundles. The morphism $S \rightarrow [Z/T]$ given by $\pi: \T \rightarrow S$ and $f: \T \rightarrow Z$ is equivalent to the data of $\T$ and the section $\sigma$ of $\T \times_T Z$ defined by the quotient of $\T \xrightarrow{(id, f)} \T\times Z$. Then $w$ acts by
\begin{equation}\label{eq:weyl2}
w\cdot (\T, \sigma: S \rightarrow \T \times_T Z) = (w\T, \varpi \circ \sigma: S \rightarrow w\T\times_T Z)\end{equation}
where $\varpi: \T \times_T Z \rightarrow w\T \times_T Z$ is the isomorphism coming from the $\twist_w$-equivariant map
\begin{equation}\label{eq:wiso}
\begin{gathered}
\T\times Z \xrightarrow{\varpi} w\T \times Z\\
(x,z) \mapsto (x, wz).
\end{gathered}
\end{equation}

\subsubsection{On $\Pic^T(Z)$}\label{sec:wonpic}
Finally, the group $W$ acts on $\Pic^T(Z)$, by which we mean the group of line bundles on $[Z/T]$, or equivalently $T$-equivariant line bundles on $Z$. Since $W$ acts on $[Z/T]$ we get an action on $\Pic^T(Z)$ by sending $\L \in \Pic([Z/T])$ to the pullback $(w^{-1})^{*}\L$. 

We translate this action to $T$-equivariant line bundles on $Z$. If $\L \rightarrow [Z/T]$ is a line bundle, the corresponding $T$-equivariant line bundle on $Z$ is the pullback $\L_1 := \L \times_{[Z/T]} Z$. We have a commuting diagram 
\begin{equation}\label{eq:wonpic1}
\begin{tikzcd}
w^{-1}(w^*\L_1) \arrow[r, "id"] \arrow[d]& w^*\L_1 \arrow[r] \arrow[d] & \arrow[d]\L_1 \\
Z \arrow[r, "id"] & wZ \arrow[r, "w"] & Z 
\end{tikzcd}
\end{equation}
where $wZ$ is $Z$ with $T$-action twisted as in \eqref{eq:weylonbunT} and $w^{-1}(w^*\L_1)$ is defined analogously (it is the scheme $w^*\L_1$ with $T$-action given by $\cdot_{w^{-1}}$). The right square is fibered and all four of its maps are $T$-equivariant, but in the left square the horizontal maps are twisted-equivariant. The quotient of either the middle or left column by $T$ gives the bundle $w^*\L \rightarrow [Z/T]$, but only the leftmost column describes a $T$-equivariant line bundle on $Z$ that is equal to $w^{-1}\cdot \L_1$. From this one sees that 
\begin{equation}\label{eq:wonpic2}
w\cdot \L_\xi = \L_{w\xi},
\end{equation}
i.e., the map $\chi(T) \rightarrow \Pic^T(Z)$ defined in \eqref{eq:inclusion} is $W$-equivariant.

\subsection{Principal bundles on $\PP^1_k$}\label{sec:bundle_ab_nonab}
Let $k$ be an algebraically closed field. We recall Grothendieck's classification of principal $G$-bundles on $\PP^1_k$, which may be read as an abelian/nonabelian correspondence theorem. 
 Let $\Bunset_G(\PP^1_k)$ denote the set of isomorphism classes of principal $G$-bundles on $\PP^1_k$. There is a natural map
\[\psi: \Bunset_T(\PP^1) \rightarrow \Bunset_G(\PP^1)\]
defined by sending a principal $T$-bundle $\T$ to $G \times_T\T$ (see \eqref{eq:principal2}). The group $W$ acts on $\Bunset_T(\PP^1_k)$ as in \eqref{eq:weylonbunT}. Moreover, for $w \in W$ and $\T \in \Bunset_T(\PP^1_k)$, the principal $G$-bundles $G \times_T \T$ and $G \times_T w\T$ are isomorphic via a map analogous to the one in \eqref{eq:wiso}.
The following theorem is due to Grothendieck; see also \cite[393]{mehta}.

\begin{theorem}[\cite{grothendieck}]\label{lem:bundle_abnonab}
The map $\psi$ induces a bijection $\Bunset_T(\PP^1_k)/W \rightarrow \Bunset_G(\PP^1_k)$.
\end{theorem}

\begin{remark}\label{rmk:bundles}
The isomorphism class of $\T$ is determined by the homomorphism $\tilde \alpha \in \Hom(\chi(T), \ZZ)$ defined by 
\[
\tilde\alpha(\xi) = \deg_{\PP^1}(\T \times_T \CC_\xi) \quad \quad \xi \in \chi(T).
\]
Hence Theorem \ref{lem:bundle_abnonab} says that elements of $\Bunset_G(\PP^1_k)$ biject with Weyl-orbits on $\Hom(\chi(T), \ZZ)$.
\end{remark}

\subsection{Cohomology of $X\sslash G$ and $X\sslash T$}
We recall an abelian-nonabelian correspondence for cohomology of $X \sslash G$ and $X\sslash T$. Recall the diagram \eqref{eq:key_diagram} 
where $j$ is an open embedding and $g$ is a fiber bundle with fiber $G/T$. The following result is well-known, but we could not find a reference in our setting.

\begin{proposition}\label{prop:chow}
The pullback $g^*$ in diagram \eqref{eq:key_diagram} induces an isomorphism
\begin{equation}\label{eq:toshow}
g^*: A_*(Z^s(G)/G)\otimes \QQ \xrightarrow{\sim} (A_*(Z^s(G)/T)\otimes \QQ)^W.
\end{equation}
An analogous statement holds for cohomology rings with coefficients in $\QQ$.
\end{proposition}
\begin{proof} We prove the statement for Chow groups. Compare the statement and its proof with \cite[Thm~10]{Br98}. Notice that $g$ factors through $Z^s(G)/T \rightarrow Z^s(G)/B$ where $B \subset G$ is a borel subgroup containing $T$. Since $Z^s(G)/T \rightarrow Z^s(G)/B$ is an affine fiber bundle, pullback induces an isomorphism on Chow groups \cite[Lem~2.2]{totaro}, and we are reduced to showing \eqref{eq:toshow} with $B$ in place of $T$. 

Let $R=Sym(\chi(T)_\QQ)$ and let $R^W_+$ be the ring of $W$-invariants of positive degree. Recall the characteristic homomorphism that sends a character of $T$ to the associated line bundle on $G/T$. Let 
\[
\phi: R/R^W_+ \xrightarrow{\sim} A_*(G/B)\otimes \QQ
\]
be the ($W$-equivariant) isomorphism induced by the characteristic homomorphism (see for example \cite[22]{Br98}), and let $\{c_i\}_{i=1}^N$ be elements of $R$ with images $[c_i]$ in $R/R^W_+$ such that $\{\phi([c_i])\}_{i=1}^N$ is a basis for $A_*(G/B)$. Then the polynomials $c_i$ also define classes $\tilde c_i$ in $A^*(Z^s(G)/B)\otimes \QQ$ (via an analogous characteristic homomorphism induced by mapping characters to associated line bundles), and the restrictions of these classes to every fiber of $g$ are a basis.

Thus we may apply the algebraic Leray-Hirsch theorem \cite[Lem~6]{EG97} to the $G/B$-fiber bundle $g$. Let $\{b_j\}$ be a basis for $A_*(Z^s(G)/G)\otimes \QQ$. This theorem says that the map $\phi: A_*(Z^s(G)/G)_\QQ \otimes A_*(G/B)_\QQ \rightarrow A_*(Z^s(G)/B)_\QQ$ given by
\[
\sum b_j \otimes \phi([c_i]) \mapsto \sum g^*(b_j)\cup \tilde c_i
\]
is an isomorphism. Because $g$ is $W$-invariant, a direct computation shows that this map is $W$-equivariant. Now we take $W$-invariants of both sides. On the left side, if $\sum b_j \otimes \phi([c_i])$ is $W$-invariant, then because the classes $b_j$ are all $W$-invariant, the classes $[c_i] \in R/R^W_+$ are as well. Hence if they are nonzero, they have degree zero. The result follows.
\end{proof}

\section{The quasimap $I$-function}\label{sec:qmap}
The $I$-function of $Z \sslash G$ can be defined using the quasimap theory of Ciocan-Fontanine and Kim, developed in \cite{wcgis0}, \cite{stable_qmaps}, and \cite{toric_qmaps}. To compute the $I$-function, we only need to study genus-zero quasimaps to $Z\sslash G$ with no markings and one parameterized component, so we will restrict our review of quasimap theory to this situation.

\subsection{Quasimaps}
Fix for the duration of this paper a copy of the projective line with projective coordinates $[u:v]$ and denote it $\specialP$. Stable graph quasimaps are defined in \cite[Def~7.2.1]{stable_qmaps}; one easily sees that when $g=0$ and there are no marked points, that definition is equivalent to the following.
\begin{definition}\label{def:qdef3}
A \textit{stable graph quasimap} to $Z\sslash G$ over a base scheme $S$ is a tuple $(C,\P, \sigma,\bx)$ where
\begin{itemize}
\item $C \rightarrow S$ is a nodal genus-0 projective curve
\item $\P \rightarrow C$ is a principal $G$-bundle
\item $\sigma$ is a section of $\P \times_G Z$
\item $\bx: C \rightarrow \specialP_S$ is a morphism that restricts to an isomorphism on every geometric fiber over $S$.
\end{itemize}
Moreover, for every geometric point $s \in S$, the set of points $ p \in C_s$ such that $\sigma(p) \not \in Z^s_\theta$ must be finite.
\end{definition}
An isomorphism of stable graph quasimaps $(C, \P, \sigma, \bx)$ and $(C, \P', \sigma', \bx')$ on $S$ is a commuting diagram
\begin{equation}\label{fig:qgmorph}
\begin{tikzcd}
&\P' \arrow[r, "\sim"] \arrow[d] &\P\arrow[d] \\
\specialP & \arrow[l,"\bx'"'] \mathcal{C}'\arrow[r,"\sim","f"']&\mathcal{C} \arrow[ll, bend left, "\bx"]
\end{tikzcd}
\end{equation}
such that the square is fibered and $f^*\sigma = \sigma'$. Note that stability depends on the character $\theta$; since we have fixed $\theta$ once and for all in this paper, we will generally omit it from the notation.
The set $\{p \in C \mid \sigma(p) \not \in Z^s\}$ is called the \textit{base locus} of the quasimap. 

\begin{remark}\label{rmk:general-qmaps}
As defined in \cite[Sec~3.1]{stable_qmaps}, the data of a quasimap over a base scheme $S$ is a tuple $(C, \P, \sigma, p_i)$ where $C \rightarrow S$ is a nodal curve of genus $g$, the principal bundle $\P$ and section $\sigma$ are as in Definition \ref{def:qdef3}, and $p_i: S \rightarrow C$ are markings for $i=1, \ldots, n$. This data is required to satisfy the same nondegeneracy condition as in Definition \ref{def:qdef3}; namely, for every geometric point $s \in S$, the set of points $p \in C_s$ such that $\sigma(p) \not \in Z^s_\theta$ must be finite. 
\end{remark}

\begin{definition}
Because we work almost exclusively with the quasimaps in Definition \ref{def:qdef3}, we call them simply ``quasimaps.'' We will refer explicitly to Remark \ref{rmk:general-qmaps} when the more general notion arises.
\end{definition}

\begin{lemma}\label{lem:trivial-curve}
If $(C, \P, \sigma, \bx)$ is a quasimap, then $\bx$ is an isomorphism (so $C = \specialP_S$).
\end{lemma}
\begin{proof}The condition that the geometric fibers of $\bx$ are isomorphisms forces $\bx$ to be an isomorphism. One can show this directly, or one can note that $(C, \bx)$ defines a family in $\Mbar_{0,0}(\PP^1, 1)$. This latter moduli space is isomorphic to the Grassmannian $Gr(\PP^1, \PP^1)$ (according to \cite[5]{FuPa}), which is represented by a point. So its universal curve is trivial.
\end{proof}

\subsubsection{Quasimaps as maps to $[Z/G]$}\label{sec:mapstostack}
Given a principal $G$-bundle $\P \rightarrow C$ on a curve $C$ over $S$, giving a section $\sigma:C \rightarrow \P \times_G Z$ is equivalent to giving a $G$-equivariant morphism
\[
\tilde n: \P \rightarrow Z.
\]
which in turn defines a map $n: C \rightarrow [Z/G]$. Indeed, a morphism $\tilde n: \P \rightarrow Z$ is equivalent to a section of 
\[\P \times Z \rightarrow \P,\] 
and the quotient of this section by $G$ recovers $\sigma$. The morphism $\tilde n$ defines a stable quasimap if for every geometric fiber $C_s$ of $C \rightarrow S$, the set of points $p \in C_s$ such that $\tilde \sigma(\P_p) \not \subset Z^s_\theta$ is finite.
The correspondence between $\sigma$ and $\tilde n$ is made more precise using the moduli of sections in Section \ref{sec:qmapmoduli}. 

\subsubsection{Quasimaps in local coordinates}\label{sec:uv-qmaps}

We will often work with a certain class of quasimaps which we now describe. Let 
\[U_S := S \times \AA^1 \xrightarrow{(s, u) \mapsto (s, [u:1])} S \times \specialP \quad \quad \quad \quad V_S := S\times \AA^1\xrightarrow{(s, v) \mapsto (s, [1:v])} S \times \specialP\] 
be the distinguished open subsets of $\specialP_S = S\times \specialP$, with gluing morphism $\kappa: U_S\setminus(S\times\{0\}) \rightarrow V_S\setminus(S\times\{0\})$ given by $\kappa(s,u) = (s, u^{-1})$. Then any morphism $\tau:U_S \setminus \{S \times 0\} \rightarrow G$ defines an isomorphism of the restrictions of the trivial $G$-bundles $U_S \times G$ and $V_S \times G$ by sending \begin{equation}\label{eq:coord4}(u, g)\mapsto(\kappa(u), g\tau^{-1}(u))\quad \text{for}\quad u \in U_S \setminus\{S\times 0\}.\end{equation} 
We denote the resulting principal $G$-bundle on $\specialP_S$ by $\P_\tau$. In particular, any cocharacter $\tau$ of $G$ defines a principal $G$-bundle $\P_\tau$ on $\specialP$, or more generally on any $\specialP_S$ by pullback. If $\T$ is a principal $T$-bundle, then the map $\tau \mapsto \T_\tau$ is $W$-equivariant with respect to the actions defined in \eqref{eq:dual_actions} and \eqref{eq:weylonbunT}.

A quasimap $(\specialP_S, \P_\tau, \sigma, id)$ may be described as follows. Due to the convention in \eqref{eq:principal1}, the fiber bundle $\P_\tau \times_G Z$ is given by gluing the trivial bundles $U_S \times Z$ and $V_S \times Z$ via 
\[(u, z) \mapsto (\kappa(u), \tau(u)z)\quad \text{for}\quad u \in U_S \setminus\{S\times 0\}.\] So $\sigma$ is determined by the maps $\sigma_U: U_S \rightarrow Z$ and $\sigma_V: V_S \rightarrow Z$, where $\sigma_U$ is the composition
\[
U_S \xrightarrow{\sigma|_{U_S}} (\P_\tau \times_G Z)|_{U_S} = U_S \times Z \xrightarrow{pr_2} Z
\]
and $\sigma_V$ is defined similarly. Hence we have
\begin{equation}\label{eq:coord1}
\tau\cdot\sigma_U = \sigma_V\circ \kappa \quad \quad\text{on}\; U_S\setminus\{S \times 0\},
\end{equation}
and conversely a pair of morphisms $\sigma_U: U_S \rightarrow Z$ and $\sigma_V: V_S \rightarrow Z$ satisfying \eqref{eq:coord1} define a section of $\P_\tau \times_G Z$.
Two quasimaps $(\specialP_S, \P_\tau, \sigma, id)$ and $(\specialP_S, \P_\omega, \rho, id)$ are isomorphic if and only if there are functions $\phi_U: U_S \rightarrow G$ and $\phi_V: V_S \rightarrow G$ such that
\begin{equation}\label{eq:coordinateiso}\begin{gathered}
(\phi_V\circ \kappa)\tau = \omega\phi_U \quad \quad \text{as maps}\; U_S\setminus (S\times\{0\})\rightarrow G\\
\phi_U\cdot\sigma_U = \rho_U \quad \quad \text{as maps}\;U_S \rightarrow Z\\
\phi_V\cdot\sigma_V = \rho_V \quad \quad \text{as maps}\;V_S \rightarrow Z.\\
\end{gathered}\end{equation}

\begin{remark}If $k$ is an algebraically closed field, then $U=V=\AA^1_k$, and since every principal $G$-bundle is trivial on $\AA^1_k$ we see that any $k$-quasimap is isomorphic to one of the form $(\specialP_k, \P_\tau, \sigma, id)$ for some transition function $\tau$. 
\end{remark}

\subsubsection{Class}

Let $k$ be an algebraically closed field. The \textit{class} of a quasimap $(\specialP_k, \P, \sigma, \bx)$ is the homomorphism $\beta \in \Hom(\Pic^G(Z), \ZZ)$ given by
    \[
    \beta(\L) = \deg_{C}(\sigma^*(\P\times_G \L)) \quad \quad \L \in \Pic^G(Z).
    \]
A family of quasimaps $(\specialP_S, \P, \sigma, \bx)$ has class $\beta$ if each of its geometric fibers over $S$ has class $\beta$.
 
 Let $T \subset G$ be a maximal torus. From the morphisms $\chi(G) \rightarrow \chi(T)$ and $\Pic^G(Z) \rightarrow \Pic^T(Z)$ and the inclusion \eqref{eq:inclusion}, we have the following diagram, crucial for understanding how class works in the abelian-nonabelian correspondence:
\begin{equation}\label{eq:degree_diagram}
\begin{tikzcd}
\Hom(\Pic^T(Z),\ZZ)\arrow[r,"\rpic"]\arrow[d, "\bdtilde"]&\Hom(\Pic^G(Z), \ZZ)\arrow[d,"\bd"]\\
\Hom(\chi(T),\ZZ)\arrow[r, "\rchi"] & \Hom(\chi(G),\ZZ))
\end{tikzcd}
\end{equation}
The maps are all given by restriction of homomorphisms.

\begin{definition}\label{def:theta-effective}
We briefly consider general quasimaps as defined in Remark \ref{rmk:general-qmaps}, with nodal, marked, possibly disconnected source curves. The classes of all such quasimaps form a semigroup, called the $\theta$-effective classes of $(Z, G)$ (see \cite[Def~3.2.2]{stable_qmaps}).
\end{definition}
Because this paper focuses on computing small $I$-functions, it will be convenient to have a name for the subset of $\theta$-effective classes that show up in the expansion of that power series. These are precisely the $I$-effective classes defined below. In contrast with the $\theta$-effective classes, one can often compute the $I$-effective classes directly (see Section \ref{sec:example}).
\begin{definition}\label{def:effective}
The classes $\beta \in \Hom(\Pic^G(Z), \ZZ)$ that are realized as the class of some stable (graph) quasimap to $Z \sslash_\theta G$ are called the \textit{I-effective classes} of $(Z,G, \theta)$.
\end{definition}
\begin{remark}\label{rmk:effective-generation}
The $I$-effective classes generate the sub-semigroup of $\theta$-effective classes defined by source curves with genus 0. Indeed, let $(C, \P, \sigma, p_i)$ be a general quasimap of class $\beta$ as in Remark \ref{rmk:general-qmaps}, such that the genus of $C$ is zero. Let $C_1, \ldots, C_N$ denote the irreducible components of $C$. Observe that by choosing any parametrization $\bx_i$ of $C_i$ and forgetting the markings, we get stable (graph) quasimaps $(C_i, \P|_{C_i}, \sigma_{C_i}, \bx_i)$ to $Z \sslash_\theta G$, and the class $\beta_i$ of $(C_i, \P|_{C_i}, \sigma_{C_i}, \bx_i)$ does not depend on the choice of parametrization $\bx_i$. Moreover, for $\L \in \Pic^G(Z)$, we have
\[
\beta(\L) = \deg_C(\sigma^*(\P \times_G \L)) = \sum_{i=1}^N \deg_{C_i}(\sigma|_{C_i}^*(\P|_{C_i}\times_G \L)) = \sum_{i=1}^N \beta_i(\L).
\]
Hence the $\theta$-effective class $\beta$ is a sum of the $I$-effective classes $\beta_i$.
\end{remark}

Let ${QG}_{\beta}(Z\sslash G)$ denote the groupoid of stable class-$\beta$ quasimaps to $Z \sslash G$. We will denote it simply $QG_\beta$ when the target $Z\sslash G$ is understood. The space ${QG}_{\beta}$ is called a \textit{quasimap graph space} in analogy with Gromov-Witten theory, and it is equal to the space $\mathrm{Qmap}_{0,0}(Z\sslash G, \beta; \PP^1)$ from \cite{stable_qmaps}.

\begin{theorem}[{\cite[Theorem~7.2.2]{stable_qmaps}}]\label{thm:qg}
The moduli space ${QG}_{\beta}(Z \sslash G)$ is a separated Deligne-Mumford stack of finite type, proper over $\bS_G$.
\end{theorem}

The following lemma is essentially proved in the proof of \cite[Theorem~7.2.2]{stable_qmaps}. Since we will use the statement, we explicitly extract its proof.

\begin{lemma}\label{lem:finite}
When $\rpic$ is restricted to $I$-effective classes in both the source and target, it has finite fibers.
\end{lemma}
\begin{proof}
By \cite[Prop~2.5.2]{stable_qmaps}, for any $Z$ there is a $G$-equivariant embedding $Z \subset V$ into a vector space $V$. The vector space $V$ may not satisfy the assumption $V^s = V^{ss}$, but for $\tilde \alpha \in \Hom(\Pic^G(V), \ZZ)$ one can still define the stack $QG_{\tilde \alpha}(V\sslash G)$, and it is shown in \cite[Thm~3.2.5]{stable_qmaps} that this stack is finite type. In fact, it is argued there that for classes that are effective for $(V, G, \theta)$, the map $\rchi$ has finite fibers. Since we may also assume $Z^s \subset V^s$, it follows that $\rchi$ has finite fibers for general $Z$.

Now, for $\tilde \alpha \in \Hom(\chi(T), \ZZ)$ let $\mathfrak{X}=\bigsqcup_{\tilde \beta\in\bdtilde^{-1}(\tilde \alpha)} QG_{\tilde \beta}(Z\sslash T)$. Then $\mathfrak{X}$ is a closed substack of $QG_{\tilde \alpha}(V\sslash T)$, hence finite type. On the other hand, the components $QG_{\tilde \beta}(Z\sslash T)$ are open in $\mathfrak{X}$ because they are loci where (flat) families of line bundles have specified degrees. Since $\mathfrak{X}$ has finite type, this shows that $\bdtilde$ also has finite fibers in effective classes.
\end{proof}

\subsection{Relation to the moduli of sections}\label{sec:qmapmoduli}
Let $QG(Z \sslash G)$ be the groupoid of stable quasimaps to $Z \sslash G$ (of any class). We describe it as a moduli of sections, using the language and notation of \cite[Appendix~A]{cjw}. This is important for our description of its obstruction theory.

Recall that if $C \rightarrow \fU$ is a family of curves and $Z \rightarrow C$ is an algebraic stack, then $\Sec{\fU}{Z}{C}$ is the groupoid whose fiber over $S \rightarrow \fU$ is $\Hom_C(C\times_{\fU}S, Z)$. Under mild assumptions, $\Sec{\fU}{Z}{C}$ is an algebraic stack \cite[Thm~1.3]{HR19} and it has a canonical obstruction theory \cite[Thm~2.1.2]{dissertation}.

\begin{example} Let $BG$ denote the global quotient $[pt/G]$. Then the stack $\Sec{pt}{\specialP \times BG}{\specialP}$ is the moduli of principal $G$-bundles on $\specialP$, which we will denote $\Bun_G$.
\end{example}

From the tower 
\begin{equation}\label{eq:tower}
\specialP \times [Z/G] \rightarrow \specialP \times BG \rightarrow \specialP \rightarrow pt
\end{equation}
we may construct the stack $\Sec{pt}{\specialP\times[Z/G]}{\specialP}$ which, when we define quasimaps as in Section \ref{sec:mapstostack}, contains $QG(Z\sslash G)$ as an open subset. This uses Lemma \ref{lem:trivial-curve}. On the other hand, the projection $\specialP \times [Z/G] \rightarrow \specialP \times BG$ induces a morphism \begin{equation}\label{eq:forget}
\Sec{pt}{\specialP \times [Z/G]}{\specialP} \rightarrow \Bun_G,
 \end{equation}
and by \cite[Lem~A.1.2]{cjw} there is a canonical isomorphism
\[
\Sec{pt}{\specialP \times [Z/G]}{\specialP} \simeq \Sec{\Bun_G}{\P \times_G Z}{\specialP_{\Bun_G}}
\]
where $\P$ denotes the universal principal bundle on $\specialP_{\Bun_G}$. Viewing $QG(Z\sslash G)$ as an open subset of $\Sec{\Bun_G}{\P \times_G Z}{\specialP_{\Bun_G}}$ produces the description of quasimaps in Definition \ref{def:qdef3}. We note that the restriction of the forgetful map \eqref{eq:forget} to $QG_\beta(Z\sslash G)$ is denoted $\mu$ in \cite[Sec~7.2]{stable_qmaps}, and the morphism $QG_\beta(Z\sslash G) \rightarrow \mathcal{M}_{0,0}(\PP^1, 1)=pt$ is denoted $\nu$.

\subsection{Perfect obstruction theory}\label{sec:pot}
Let $\pi: QG_\beta(Z\sslash G) \times \specialP \rightarrow QG_\beta(Z\sslash G)$ be the universal curve and $n: QG_\beta(Z\sslash G) \times \specialP \rightarrow [Z/G]$ the universal map. The perfect obstruction theory used in \cite[Sec~7.2]{stable_qmaps} and \cite[Sec~2.6]{wcgis0} is the canonical one coming from the inclusion $QG_\beta(Z\sslash G) \subset \Sec{\Bun_G}{\P\times_G Z}{\specialP_{\Bun_G}}$. This theory is a relative one over $\Bun_G$. We prefer to use the canonical obstruction theory coming from the inclusion $QG_\beta(Z\sslash G) \subset \Sec{pt}{\specialP \times [Z/G]}{\specialP}$ because this one is absolute. That is, we define
\begin{equation}\label{eq:pot}
\phi: \EE_{QG_\beta}:=R \pi_*n^*(\LL_{[Z/G]}\otimes \omega^\bullet) \rightarrow \LL_{QG_\beta}
\end{equation}
to be the absolute obstruction theory on $QG_\beta$, where $\phi$ is defined as in \cite[(18)]{cjw} (see also \cite[Sec~2.1.1]{dissertation}). We note that since the cotangent complex $\LL_{[Z/G]}$ is perfect, there is a canonical isomorphism
\[
\EE_{QG_\beta} \simeq (R\pi_*n^*\TT_{[Z/G]})^\vee
\]
given by \cite[(4.1)]{FHM} (it is an isomorphism by \cite[Thm~4.4]{FHM} and \cite[Prop~2.2.6]{dissertation}).
A priori, this $\phi$ may not agree with the $\nu$-relative theory defined in \cite[Sec~7.2]{stable_qmaps}, which is an absolute obstruction theory defined via a mapping cone construction to be compatible with the canonical $\Bun_G$-relative theory. 

\begin{lemma}\label{lem:pot}
The arrow \eqref{eq:pot} is an absolute perfect obstruction theory on $QG_\beta$ inducing the same virtual cycle as the one used in \cite{stable_qmaps}.
\end{lemma}
\begin{proof}
The argument in \cite[Sec~A.2.3]{cjw} shows that \eqref{eq:pot} is in fact compatible with the $\mu$-relative theory in \cite{stable_qmaps}, implying that $\phi$ is a \textit{perfect} obstruction theory and that it induces the same virtual cycle as the $\mu$- and $\nu$-relative theories in \cite{stable_qmaps}.
\end{proof}

\begin{remark}
In this paper we need an absolute obstruction theory on $QG_\beta$ in order to compute the virtual fundamental classes of the fixed loci in the definition of the $I$-function (Section \ref{sec:I}). The $\nu$-relative theory of \cite{stable_qmaps} is an absolute theory, but because it is defined as a non-canonical mapping cone, we were not able to show that it is functorial under abelianization.
\end{remark}

\subsection{I-function}\label{sec:I}
Let $\CC^*$ act on $\PP^1$
by
\begin{equation}\label{eq:action1}
\lambda \cdot [u:v] = [\lambda u:v], \quad \quad\quad\quad \lambda \in \CC^*.
\end{equation}
This induces an action on $QG_\beta$ via 
\begin{equation}\label{eq:action2}\lambda \cdot (C, \P, \sigma, \bx) = (C, \P,\sigma, \lambda \circ \bx).\end{equation}
Every quasimap is isomorphic to one with $\bx=id$. For these quasimaps, the action in \eqref{eq:action2} is equivalent to setting
\begin{equation}\label{eq:action3}
\lambda \cdot (\specialP_S, \P, \sigma, id) = (\specialP_S, (\lambda^{-1})^*\P,\sigma\circ \lambda^{-1}, id).
\end{equation}




In terms of the moduli of sections, the action described on $QG_\beta$ comes from the $\CC^*$-equivariant structure on the tower of morphisms \eqref{eq:tower} given by letting $\CC^*$ act on $\specialP$ via \eqref{eq:action1}. By \cite[Sec~A.3]{cjw} this equivariant tower induces $\CC^*$-actions on $QG_\beta$ and $C_{QG_\beta}$ making $\pi$ and $n$ equivariant. It also induces a canonical $\CC^*$-equivariant structure on the perfect obstruction theory \eqref{eq:pot}.


We define the \textit{fixed locus} of $QG_\beta$ under the $\CC^*$-action as in \cite[Sec~3]{CKL17}. Its closed points are geometric quasimaps $(\specialP_k, \P, \sigma, \bx)$ such that $\lambda \cdot (\specialP_k, \P, \sigma, \bx)$ is isomorphic to $(\specialP_k, \P, \sigma, \bx)$ for every $\lambda \in \CC^*$ (see eg \cite[Prop~5.23]{AHR}). The $I$-function of $Z\sslash G$ is defined in terms of localization residues at certain fixed loci (see \cite[Sec~7.3]{stable_qmaps}). Observe that a fixed quasimap must have all its base points at $[0:1]$ or $[1:0]$, and that the resulting map 
\begin{equation}\label{eq:constant}
\PP^1\setminus\{[1:0], [0:1]\} \rightarrow Z \sslash G
\end{equation}
must be constant. Let $F_{\beta}(Z\sslash G)$ denote the component of the fixed locus of $QG_{\beta}(Z\sslash G)$ corresponding to quasimaps that have a unique basepoint at $[0:1]$ (it may not be connected). We will omit the space $Z\sslash G$ from the notation when there is no danger of confusion. Let $ev_{\bullet}:F_{\beta} \rightarrow Z\sslash G$ send a quasimap to the point in $Z\sslash G$ that is the image of the constant map \eqref{eq:constant}. Then we can define the $I$-function of $Z\sslash_{\theta}G$ as a formal power series in the $q$-adic completion of the semigroup ring generated by the semigroup of $\theta$-effective classes.

\begin{definition}
The (small) $I$-function of $Z\sslash_{\theta} G$ is
\begin{equation}\label{eq:Ifunc}
I^{Z\sslash G}(z) = 1+ \sum_{\beta\neq 0}q^\beta I^{Z\sslash G}_\beta(z) \quad \quad \text{where} \quad \quad I^{Z\sslash G}_\beta(z) = (ev_{\bullet})_*\left(\frac{[F_\beta(Z\sslash G)]^{vir}}{e_{\CC^*}(N^{vir}_{F_{\beta}(Z\sslash G)})}\right)
\end{equation}
and the sum is over all $I$-effective classes of $(Z, G, \theta)$.
  \end{definition}
\begin{remark}
In \eqref{eq:Ifunc} it is equivalent to sum over all $\theta$-effective classes.
\end{remark}

\subsection{Interpreting the main theorem}\label{sec:interpret}
We discuss how to interpret the right hand side of \eqref{eq:main}. 
\begin{lemma} The series
\begin{equation}\label{eq:rhs}
\sum_{\tilde \beta \rightarrow \beta} \left( \prod_{\rho} \frac{\prod_{k=-\infty}^{\tilde\beta(\rho)}(c_1(\L_{\rho}) + kz)}{\prod_{k=-\infty}^0 (c_1(\L_{\rho}) + kz)}\right)I^{Z\sslash T}_{\tilde \beta}(z)\end{equation}
appearing on the right hand side of \eqref{eq:main} is invariant under the action of $W$.
\end{lemma}
\begin{proof}
We first claim that for $w \in W$, we have $(w^{-1})^*I^{Z\sslash T}_{\tilde \beta}(z) = I^{Z\sslash T}_{w\tilde \beta}(z).$ To see this, let $QG(Z\sslash T)$ denote the moduli of quasimaps to $Z\sslash T$ (of any class). The $W$-action on $[Z/T]$ induces an action on $QG(Z\sslash T)$ that is compatible with the action on classes and makes $ev_\bullet$ $W$-equivariant (see Lemma \ref{lem:weyl_action}). Hence we have
\[(w^{-1})^*(ev_\bullet)_*=(ev_\bullet)_*(w^{-1})^*\]
where the evaluation map on the left hand side is for $F_{\tilde \beta}(Z\sslash T)$ and on the right hand side it is for $F_{w\tilde \beta}(Z\sslash T).$
(The situation in \eqref{eq:twosquares} is almost identical).

Moreover, by \cite[Sec~A.3]{cjw}, the obstruction theory on $QG(Z\sslash T)$ is $W$-equivariant. This means that $(w^{-1})^*\EE_{QG}$ is canonically isomorphic to $\EE_{QG}$, and in particular they have the same localization residue on $F_{w\tilde \beta}(Z\sslash T)$. 

From here, a straightforward computation using the formulas in Section \ref{sec:weylaction} completes the proof.
\end{proof}

Because of this lemma, we may interpret the $W$-invariant coefficients of \eqref{eq:rhs} as elements of $A_*(Z\sslash G)_\QQ$ using Proposition \ref{prop:chow}. 

However, the reader may still wonder about the denominators that arise when $\tilde \beta(\rho)<0$. In fact, if we fix a set of positive roots $R^+$ with respect to $T$, then the class $\Delta = \prod_{\rho \in R^+} c_1(\L_\rho)$ appears in the denominator of every summand of \eqref{eq:rhs}. This class is the fundamental $W$-anti-invariant class, and it plays a role analogous to the Vandermonde determinant for symmetric functions, namely, every $W$ anti-invariant class is divisible by $\Delta$, and the quotient is $W$-invariant. These facts are explained in \cite[Sec~1]{ellingsrud}.

\section{Relate the fixed loci and evaluation maps}
The goal of this section is to ``pull back'' diagram \eqref{eq:key_diagram} to the $\CC^*$-fixed loci in the quasimap moduli spaces. That is, we prove the following.


\begin{proposition}\label{prop:diagram}
Let $\beta \in \Hom(\Pic^G(Z),\ZZ)$ be $I$-effective. For every $\tilde \beta \in \rpic^{-1}(\beta)$, there is
\begin{itemize}
\item a parabolic subgroup $P_{\bdtilde(\tilde \beta)}\subset G$, and
\item a morphism $\psi_{\tilde \beta}: F_{\tilde \beta}(Z\sslash T) \cap Z^s(G) \rightarrow F_\beta(Z\sslash G)$ whose image we denote $F_{\tilde \beta}(Z\sslash G)$,
\end{itemize}
fitting into the following commutative diagram:
\begin{equation}\label{eq:big_diagram}
\begin{tikzcd}
 F_{\tilde \beta}(Z\sslash G) \arrow[d,hook,"i"]  & F_{\tilde \beta}(Z\sslash T) \cap Z^s(G) \arrow[l,"\psi_{\tilde \beta}"'] \arrow[d,hook]\arrow[r,hook,"h"]& F_{\tilde \beta}(Z\sslash T)\arrow[d,hook, "ev_\bullet"]\\
Z^s(G)/P_{\bdtilde(\tilde \beta)}  \arrow[dr, "f"]& Z^s(G)/T \arrow[r,hook,"j"]\arrow[d, "g"]\arrow[l,"p"]& Z^s(T)/T\\
& Z^s(G)/G
\end{tikzcd}
\end{equation}
Here, the two squares are fibered, the vertical arrows in the top row are all closed embeddings, and the composition $f\circ i$ is the evaluation map $ev_\bullet$.
\end{proposition}


\subsection{Preliminaries, including definition of $P_{\bdtilde(\tilde \beta)}$}
We may identify cocharacters with dual characters of $T$ as follows. A dual character $\tilde \alpha \in \Hom(\chi(T),\ZZ)$ determines a cocharacter $\tau_{\tilde \alpha}$ via the rule
\begin{equation}\label{eq:dual}
\xi(\tau_{\tilde \alpha}(t)) = t^{-\tilde \alpha(\xi)} \quad\quad \text{for any}\quad \quad\xi \in \chi(T).
\end{equation}
The negative sign in the exponential appears so that for $\xi \in \chi(T)$ we have \[\deg_{\PP^1}(\T_{\tau_{\tilde \alpha}} \times_T \CC_\xi) = \tilde \alpha(\xi)\] (so in particular, if $Z$ is a vector space, a quasimap to $Z\sslash T$ with principal bundle $\T_{\tau_{\tilde \alpha}}$ has class $\tilde \alpha$). One can check that this identification of cocharacters and dual characters is $W$-equivariant under the actions defined in \eqref{eq:dual_actions}. To lighten the notation we will write $\T_{\tilde \alpha}$ for $\T_{\tau_{\tilde \alpha}}$ and $\P_{\tilde \alpha}$ for the associated principal $G$-bundle (which is equal to $\P_{\tau_{\tilde \alpha}}$).

\begin{remark}\label{rmk:bundle}
Let $k$ be an algebraically closed field. By the classification of principal bundles on $\PP^1_k$, every $k$-point of $QG_\beta$ is represented by a quasimap of the form $(\PP^1_k, \P_{\tilde \alpha}, \sigma, id)$ where $\rchi(\tilde \alpha) = \bd(\beta)$ (see Theorem \ref{lem:bundle_abnonab} and Remark \ref{rmk:bundles}).
\end{remark}

The construction of the parabolic subgroup $P_{\bdtilde(\tilde \beta)}$ uses the ``dynamic method'' (see e.g. \cite[Sec~2.1]{cgp10}). If $\tilde \alpha = \bdtilde(\tilde \beta)$ is a dual character and $\tau_{\tilde \alpha}$ the cocharacter defined in \eqref{eq:dual}, then the dynamic method defines a parabolic subgroup whose points over a $G$-scheme $S$ are
\begin{equation}\label{eq:defP}
P_{\tilde \alpha}(S) = \{ g \in G(S) \mid \lim_{t \rightarrow 0} \tau_{\tilde \alpha}(t)^{-1} g \tau_{\tilde \alpha}(t)\;\text{exists in}\;G\}.
\end{equation}
By ``the limit exists in $G$'' we mean that there is a dotted arrow making the following diagram of $S$-schemes commute:
\begin{equation}\label{eq:dynamic}
\begin{tikzcd}
\CC^*_S \arrow[d, hook] \arrow[r, "\tau_{\tilde \alpha}(t)^{-1}g\tau_{\tilde \alpha}(t)"] & [45pt]G_S\\
\AA^1_S \arrow[ur, dashrightarrow]
\end{tikzcd}
\end{equation}
By considering the case when $S$ is affine (so all schemes in the diagram are also affine), one sees that the dotted arrow is unique if it exists.
The subgroup $P_{\tilde \alpha}$ clearly contains $T$. 
\begin{remark}\label{rmk:parabolic-auts}The group $P_{\tilde \alpha}$ has a natural inclusion into $\Aut(\P_{\tilde \alpha})\subset \Aut(\P_{\tilde \alpha}\times_G Z)$, given by sending $g\in P_{\tilde \alpha}$ to the automorphism defined as in \eqref{eq:coordinateiso} by setting $\phi_V(v)=g$ and setting $\phi_U$ to be the unique dotted arrow in \eqref{eq:dynamic}.
\end{remark}
The dynamic method also produces a canonical Levi subgroup $L_{\tilde \alpha} \subset P_{\tilde \alpha}$, equal to the centrilizer of $\tau_{\tilde \alpha}$:
\[
L_{\tilde \alpha} = \{g \in G \mid \tau_{\tilde \alpha}(t)^{-1}g\tau_{\tilde \alpha}(t) = g\} \]
In fact this is the unique Levi sugbroup of $P_{\tilde \alpha}$ containing $T$ (see \cite[Prop~12.3.1]{confeng}).

We close this section with some properties of $F_\beta(Z\sslash G)$.

\begin{lemma}\label{lem:properties}
The stack $F_{\beta}(Z\sslash G)$ is represented by a separated algebraic space, and it is proper over $\bS_G$.
\end{lemma}
\begin{proof}
From the definition of torus fixed loci in \cite[Sec~3]{CKL17} we see that $F_\beta(Z\sslash G)$ is a closed substack of $QG_\beta$, hence proper and separated by Theorem \ref{thm:qg}.

To see that its automorphism groups are trivial, let $(\specialP_S, \P, \sigma, id)$ be a quasimap in $F_\beta$ over a scheme $S$, and let $\phi$ be an automorphism of it, i.e., $\phi$ is an automorphism of $\P$ such that the induced automorphism of $\P \times_G Z$ fixes $\sigma$. If $U \rightarrow \specialP_S$ is an \'etale chart where $\P$ is trivial, then $\sigma$ is given by a map $\sigma_U: U \rightarrow Z$ and $\phi$ is given by $\phi_U:U \rightarrow G$, and these data satisfy
\[
\phi_U(u)\sigma_U(u)=\sigma_U(u)
\]
for each $u \in U$. This means $\phi_U(u)$ is in the stabilizer $G_{\sigma_U(u)}$. Because the quasimap is stable, the group $G_{\sigma_U(u)}$ is trivial on an open subset of $U$. Hence $\phi_U$ is the identity, and $\phi$ is trivial.
\end{proof}

\subsection{Abelian case}
The goal of this section is to prove the following.

\begin{lemma}\label{lem:abelian}
The map $ev_\bullet: F_{\tilde \beta}(Z\sslash T) \rightarrow Z^s(T)/T$ is a closed embedding.
\end{lemma}

At the end of the section, we use Lemma \ref{lem:abelian} to describe the universal family on $F_{\tilde \beta}(Z\sslash G)$ (Proposition \ref{prop:ufam}).
The results in this section are related to those in \cite[Sec~5.2]{orb-qmaps}.
We begin with three lemmas, the first two of which are probably standard.

\begin{lemma}\label{lem:bundle-extensions}
Let $S \hookrightarrow S'$ be a square-zero extension of schemes. If $\P_1$ and $\P_2$ are two principal $G$-bundles on $\specialP_{S'}$ such that $\P_1|_{\specialP_S} \simeq \P_2|_{\specialP_S}$, then $\P_1 \simeq \P_2$.
\end{lemma}
\begin{proof}
This follows from \cite[Thm~1.5]{olsson-deftheory} and the fact that $\LL_{[\bullet/G]}$ is represented by a vector bundle in degree 1.
\end{proof}

\begin{lemma}\label{lem:uni-inv}
Let $X$, $Y$ be algebraic spaces over $\CC$, locally of finite type. Let $\pi: X \rightarrow Y$ be a separated morphism that is injective on $\CC$-points. Then $\pi$ is universally injective.
\end{lemma}
\begin{proof}
We show that the diagonal $\Delta_{\pi}: X \rightarrow X \times_YX$ is surjective. Because $\pi$ is separated, $\Delta_{\pi}$ is closed, so the complement of the image $|\Delta_{\pi}|^C$ is an open subset of $|X \times_YX|$ and by \cite[Tag~06G2]{tag} it contains a $\CC$-point if it is nonempty. So it suffices to show that $\Delta_{\pi}$ is surjective on $\CC$-points. But if $(x_1, x_2) \in (X\times_Y X)(\CC)$, then $\pi(X_1)=\pi(x_2)$ so $x_1=x_2$. So $(x_1, x_2) = \Delta_{\pi}(x_1, x_1)$ as desired.
\end{proof}

We will use the description of a quasimap as a tuple $(\specialP_S, \T, \tilde n, id)$ where $\tilde n: \T \rightarrow Z$ is a $T$-equivariant map. Let $\star=[1:1]$ in $U \cap V \subset \specialP_\CC$ and let $\iota_\star: S \rightarrow \specialP_S$ be the section with constant image $\star$, with $\T_S := \iota_\star^*\T$. Then $ev_\bullet$ is represented by the map $ev_\bullet: F_\beta(Z \sslash G) \rightarrow [Z^s(T)/T]$ that sends $q_i$ to the map $S \rightarrow [Z^s(T)/T]$ given by  
\begin{equation}\label{eq:injective0}
\T_S \rightarrow \T \xrightarrow{\tilde n} Z.
\end{equation}

\begin{lemma}\label{lem:injective1}
If $q=(\specialP_S, \T, \sigma, id)$ is an $S$-point of $F_{\tilde \beta}(Z\sslash T)$, then we have a commuting diagram as below, with the square fibered:
\[
\begin{tikzcd}
\T|_{V\times S} \arrow[r] \arrow[d] \arrow[rr, bend left, "\tilde n"]& \T_S \arrow[d] \arrow[r, "\text{\eqref{eq:injective0}}"'] & Z\\
V \times S \arrow[r] & S
\end{tikzcd}
\]
\end{lemma}

\begin{proof}
If we replace $V$ with $\CC^* \subset V$, this is just a restatement of the fact that the morphism $n:\specialP_S \rightarrow [Z/T]$ defined by $q$ factors through the quotient $\specialP_S \rightarrow [\specialP_S/\CC^*]$, and hence $n|_{\CC^*_S}$ is equivalent (2-isomorphic) to $\CC^*_S \rightarrow S \xrightarrow{n|_\star} [Z/T]$. Hence $n$ and $V_S \rightarrow S \xrightarrow{n|_\star} [Z/T]$ agree on the open subset $\CC^*_S \subset V_S$, so since $n|_{V_S}$ factors through the separated substack $Z\sslash T \subset[Z/T]$, they agree on all of $V_S$. This translates to the desired diagram.
\end{proof}

\begin{lemma}\label{lem:injective}
Let $q_i = (\specialP_S, \T, \sigma_i, id)$ for $i=1, 2$ be two quasimaps in $F_{\tilde \beta}(Z\sslash T)$ with the same base $S$ and principal $T$-bundle $\T$. If $ev_\bullet(q_1)=ev_\bullet(q_2)$, then $q_1\simeq q_2$.
\end{lemma}
\begin{proof}
Suppose we have two quasimaps $q_i = (\specialP_S, \T, \tilde n_i, id)$ for $i=1, 2$ with $ev_\bullet(q_1)=ev_\bullet(q_2)$. Then we have an automorphism of $\T_S$ sending $\tilde n_1|_{\T_S}$ to $\tilde n_2|_{\T_S}$. This automorphism is given by a morphism $\phi: S \rightarrow T$ (see for example \cite[Prop~2.11]{balaji}). Then the composition $\specialP_S \rightarrow S \rightarrow T$ defines an element of $T(\specialP_S)$, hence an automorphism $\Phi$ of $\T$. It remains to check that $\tilde n_1\circ \Phi$ and $\tilde n_2$ define the same map from $\T$ to $Z$, given that their restrictions to $\T_S$ agree. 
Lemma \ref{lem:injective1} implies that $\tilde n_1 \circ \Phi$ and $\tilde n_2$ agree on the open subset $\T|_{V\times S} \subset \T$. Since $Z$ is separated and the complement of $\T|_{V\times S}$ in $\T$ is an effective Cartier divisor, they agree on all of $\T$ (see \cite[10.2.G]{vakil}).
\end{proof}

\begin{proof}[Proof of Lemma \ref{lem:abelian}]
By Lemma \ref{lem:properties}, $F_{\tilde \beta}(Z\sslash T)$ is an algebraic space, so we can use \cite[Tag 05W8]{tag}: it is enough to show that $ev_\bullet$ is proper, formally unramified, and universally injective.

The map $ev_\bullet$ is proper because $F_{\tilde \beta}(Z\sslash T)$ is proper over $\bS_T$ and $Z\sslash T$ is separated. It is universally injective by Remark \ref{rmk:bundle} and Lemma \ref{lem:uni-inv}. Finally, to check that it is formally unramified, let $S \hookrightarrow S'$ be a square-zero extension of schemes fitting into a solid diagram
\[
\begin{tikzcd}
S \arrow[r, "q"] \arrow[d, hook] & F_{\tilde\beta}(Z\sslash T) \arrow[d, "ev_\bullet"] \\
S' \arrow[r] \arrow[ur, dashrightarrow] & Z\sslash T
\end{tikzcd}
\]
Suppose we have two dotted arrows $q_1, q_2$ making the diagram above commute. The arrows $q_1$, and $q_2$ define principal $T$-bundles $\specialP_{S'}$ that agree after restriction to $\specialP_S$, so by Lemma \ref{lem:bundle-extensions} the two principal bundles on $\specialP_{S'}$ are isomorphic. Now Lemma \ref{lem:injective} shows $q_1=q_2$.
\end{proof}

Define
\[
Z_{\tilde \beta} = Z^s(T) \times_{Z\sslash T} F_{\tilde\beta}(Z\sslash T).
\]
Observe that $Z_{\tilde\beta}\rightarrow F_{\tilde\beta}(Z\sslash T)$ is a principal $T$-bundle (meaning it is represented by such) with a $T$-equivariant map to $Z^s(T)$ (this map is a closed embedding by Lemma \ref{lem:abelian}), and in fact this data defines the evaluation morphism $ev_\bullet: F_{\tilde\beta}(Z\sslash T) \rightarrow [Z^s(T)/T]$ (see \eqref{eq:injective0}). Let
\[
\tilde\alpha = \bdtilde(\tilde\beta).
\]

\begin{proposition}\label{prop:ufam}
The universal family on $F_{\tilde\beta}(Z\sslash T)$ has fiber bundle $\cZ$ on $\specialP_{F_{\tilde\beta}(Z\sslash T)}$ and section $\csigma$ defined as follows:
\begin{equation}\label{eq:ufam1}
\cZ = \frac{Z_{\tilde\beta} \times (\CC^2\setminus\{0\}) \times Z}{(x, \bu, y) \sim (tx, s\bu, \tau_{\tilde \alpha}(s)^{-1}ty)} \quad \quad \csigma(x,\bu) = (x, \bu, \tau_{\tilde \alpha}(u)^{-1}x)
\end{equation}
where $(x,\bu,y) \in Z_{\tilde\beta} \times (\CC^2\setminus \{0\}) \times X$ with $\bu=(u, v)$ and $(t, s) \in T \times \CC^*$.
\end{proposition}

\begin{remark}\label{rmk:def-csigma}
The section $\csigma$ is a priori defined only for $u \neq 0$, but we will see in the proof of Proposition \ref{prop:ufam} that it has a unique extension over all of $\specialP_{F_{\tilde\beta}(Z\sslash T)}$.
\end{remark}

\begin{remark}\label{rmk:tfam}
We will often use the tautological family on $Z_{\tilde \beta}$ that is the pullback of \eqref{eq:ufam1}. It is given by the same formulas as in \eqref{eq:ufam1} but without dividing by the $T$-action.
The benefit of studying this family on $Z_{\tilde\beta}$ is that it is of the form in Section \ref{sec:uv-qmaps}: its underlying principal bundle is $\T_{\tilde\alpha}$, as can be shown by computing its transition function, and we have
\begin{align*}
\csigma_U: \AA^1\times Z_{\tilde\beta} &\rightarrow Z & \csigma_V: \AA^1\times Z_{\tilde\beta} &\rightarrow Z\\
(u,z) &\mapsto \tau_{\tilde\alpha}(u)^{-1} z& (v,z) &\mapsto z
\end{align*}
where $\csigma_U$ is defined as explained in Remark \ref{rmk:def-csigma}.
\end{remark}
\begin{proof}[Proof of Proposition \ref{prop:ufam}]
We show that the tautological family on $Z_{\tilde\beta}$ is as described in Remark \ref{rmk:tfam}. Then \eqref{eq:ufam1} defines a family of fixed quasimaps on $F_{\tilde \beta}(Z\sslash T)$ that is sent by $ev_\bullet$ to the inclusion $[Z_{\tilde \beta}/T] \subset [Z^s(T)/T]$. Since $ev_\bullet$ is a closed embedding, \eqref{eq:ufam1} must be the universal family.

Let $(\specialP_{Z_{\tilde \beta}}, \T, \sigma, id)$ be the tautological family. We first show that $\T$ is isomorphic to $\T_{\tilde \alpha}$. Using the description of the evaluation map $ev_\bullet$ in the proof of Lemma \ref{lem:injective}, we see that $\T_{Z_{\tilde \beta}} = \iota_\star^*\T$ is canonically isomorphic to $Z_{\tilde \beta} \times_{F_{\tilde \beta}(Z\sslash T)} Z_{\tilde \beta}$, with the map to the base $Z_{\tilde \beta}$ equal to one of the projections. This principal $T$-bundle is trivial (it has the diagonal section). From Lemma \ref{lem:injective1} we see that $\T|_{\CC^*\times{Z_{\tilde \beta}}}$ is also trivial.

We show that $\T|_{U\times Z_{\tilde \beta}}$ and $\T|_{V\times Z_{\tilde \beta}}$ are trivial. Let $\{S_i\} \rightarrow Z_{\tilde \beta}$ be an affine open cover. If $S_i^{red}$ is the reduced subscheme of $S_i$, then since $S_i$ is Noetherian, the containment $S_i^{red} \subset S_i$ may be factored as a finite sequence of square-zero extensions. So by Lemma \ref{lem:bundle-extensions} the restriction map $\Pic(S_i\times U) \rightarrow \Pic(S_i^{red} \times U)$ is injectve. Since $S_i^{red} \times U$ is a reduced Noetherian affine scheme, by \cite{Isch} the restriction map $\Pic(S_i^{red} \times U) \rightarrow \Pic(S_i^{red} \times \CC^*)$ is also injective. This implies that the restriction $\Pic(S_i \times U) \rightarrow \Pic(S_i \times \CC^*)$ is injective. Since $\T|_{\CC^*\times S_i}$ is trivial, this implies $\T|_{U \times S_i}$ is trivial.
On the other hand, the transition function for $\T$ on $\CC^*\times (S_i \cap S_j)$ is constant and equal to the identity, since $\T_{\CC^*\times Z_{\tilde \beta}}$ is trivial. So the transition function for $\T$ on each $U \times (S_i \cap S_j)$ is trivial, and we conclude that $\T|_{U \times Z_{\tilde \beta}}$ is trivial. Likewise $\T|_{V\times Z_{\tilde \beta}}$ is trivial.

Finally we compute the transition function $\tau: \CC^*_{Z_{\tilde \beta}} \rightarrow T$ satisfying \eqref{eq:coord4} for $\T$. The morphism $\tau$ is given by a ring map $\Gamma(T, \OO_T) \rightarrow \Gamma(\CC^*_{Z_{\tilde \beta}}, \OO_{\CC^*\times Z_{\tilde \beta}})$. If we choose a basis $\xi_1, \ldots, \xi_N$ of characters of $T$, then $\tau$ is determined by a collection $p_1, \ldots, p_N$ of invertible elements of $\Gamma(Z_{\tilde \beta}, \OO_{Z_{\tilde \beta}})[u, u^{-1}]$ ($p_j$ is the image of $\xi_j$). Since geometric fibers of $\T$ have class $\tilde\alpha$, the restriction of $p_j$ to every geometric fiber is a monomial of degree $-\tilde\alpha(\xi_j)$. So the restriction of $p_j$ to $Z_{\tilde \beta}^{red}$ is also a monomial of degree $-\tilde\alpha(\xi_j)$. Changing the trivialization on $V_{Z_{\tilde \beta}}$ by the appropriate element of $T$ and recalling the relationship \eqref{eq:dual}, we can assume $\tau|_{Z_{\tilde \beta}^{red}} = \tau_{\tilde\alpha}$. By Lemma \ref{lem:bundle-extensions} the bundle $\T|_{\PP^1_{Z_{\tilde \beta}}}$ is also isomorphic to $\T_{\tilde\alpha}$ as claimed.

The second step is to show that the tautological section $\csigma$ is given by the formulae for $\csigma_U$ and $\csigma_V$ in Remark \ref{rmk:tfam}. Because evaluation is tautological, $\csigma_V$ sends $(1,z)$ to $z$. By Lemma \ref{lem:injective1}, the function $\sigma_V$ is pulled back from this fiber. A priori the pullback map is not unique, and hence $\sigma_V$ may not be completely determined; however, any two choices for $\sigma_V$ would differ by an element of $T(V_{Z_{\tilde \beta}})$, which is given by a collection of invertible elements in $\Gamma(Z_{\tilde \beta}, \OO_{Z_{\tilde \beta}})[v]$. Since these are constant with respect to $v$, we see that the only option for $\csigma_V$ is the map $\csigma_V(v, z) = z$. Then by \eqref{eq:coord1} we see that $\sigma_U(u, z) = \tau_{\tilde \alpha}(u)^{-1}z$ for $u \neq 0$, and in particular this map has an extension to all of $U_{Z_{\tilde \beta}}.$ Uniqueness of the extension may be checked affine locally on $Z_{\tilde \beta}$, since the ring map $\Gamma(U \times \Spec(A), \OO_{U \times \Spec(A)}) \rightarrow \Gamma(\CC^*\times \Spec(A), \OO_{\CC^*\times\Spec(A)})$ is injective.
\end{proof}

\subsection{Proof of Proposition \ref{prop:diagram}}
Let 
\[
\begin{gathered}F^0_{\tilde \beta}(Z \sslash T) := \Big(F_{\tilde \beta}(Z\sslash T) \cap Z^s(G)/T\Big) \subset Z^s(T)/T,\\
Z_{\tilde \beta}^0 := \Big(Z_{\tilde \beta} \cap Z^s(G) \Big)\subset Z^s(T)
\end{gathered}\] 
so $F^0_{\tilde \beta}(Z \sslash T)$ is the open substack of $F_{\tilde \beta}(Z\sslash T)$ where $ev_\bullet$ lands in $Z^s(G)/T$ and $Z_{\tilde \beta}^0$ is the natural $T$-torsor on it. 

\begin{lemma}\label{lem:Zbeta-invariance}
The subscheme $Z_{\tilde \beta}^0 \subset Z^s(G)$ is invariant under the action of $P_{\tilde\alpha} \subset G$ on $Z^s(G)$.
\end{lemma}
\begin{proof} By the definition of the semi-stable locus, there is a cover of $Z^s(G)$ by $G$-invariant affine subschemes. Passing to this cover, we may assume $Z^s(G)$ is affine, and hence it suffices to check that for $p : Z_{\tilde \beta}^0 \rightarrow P_{\tilde\alpha}$ we have $pZ_{\tilde \beta}^0\subset Z_{\tilde \beta}^0$ (i.e., global sections are invariant). Since the entire set $P_{\tilde\alpha} Z_{\tilde \beta}^0$ is $T$-invariant, if we let $(pZ_{\tilde \beta}^0)/T$ denote the quotient of the $T$-orbits meeting $pZ_{\tilde \beta}^0$, it is in fact sufficient to show that $(pZ_{\tilde \beta}^0)/T \subset Z_{\tilde \beta}^0/T = F^0_{\tilde \beta}(Z\sslash T)$. 

To do this let $(\specialP_{Z_{\tilde\beta}^0}, \T_{\tilde\alpha}, \csigma, id)$ be the tautological family on $Z_{\tilde \beta}^0$ defined by \eqref{eq:ufam1}, and observe that $\T_{\tilde \alpha}\times_T Z = \P_{\tilde \alpha} \times_G Z$. Let $\wp \in \Aut(\T_{\tilde\alpha} \times_T Z)$ be the automorphism defined by $p$ as in Remark \ref{rmk:parabolic-auts}. We claim that 
\begin{equation}\label{eq:Zbeta-invariance1}
\cF = (\specialP_{Z_{\tilde \beta}^0}, \T_{\tilde\alpha}, \wp \circ \csigma, id)
\end{equation}
is another family of $\CC^*$-fixed quasimaps on $Z^0_{\tilde \beta}$ of class $\tilde \beta$.\footnote{
Using the presentation for $\T_{\tilde \alpha} \times_T Z$ in \eqref{eq:ufam1}, the automorphism $\wp$ is given by $(x,\bu,z) \rightarrow (x, \bu, \tau_{\tilde\alpha}(u)^{-1}p\tau_{\tilde\alpha}(u) z)$ and $\wp \circ \csigma(x,\bu) = (x, \bu, \tau_{\tilde\alpha}(u)^{-1}px).$
} Granting this, its evaluation map $ev_{\bullet,\cF}: Z_{\tilde \beta}^0 \rightarrow Z\sslash T$ must factor through the universal one, namely the inclusion $F_{\tilde\beta}(Z\sslash T) \subset Z\sslash T$. But by construction the image of $ev_{\bullet, \cF}$ is precisely $pZ_{\tilde \beta}^0/T$ (note that $\csigma_V: Z_{\tilde \beta}^0 \times\{u\neq 0\} \rightarrow Z_{\tilde \beta}^0$ is the projection). So we have $(pZ_{\tilde \beta}^0)/T \subset Z_{\tilde \beta}/T$. Since the image of $ev_{\bullet,\cF}$ is also contained in $Z^s(G)/T$, we have $(pZ_{\tilde \beta}^0)/T \subset Z_{\tilde \beta}^0/T$ as desired.

That the family $\cF$ is $\CC^*$-fixed follows from the definition of the $\CC^*$-action \eqref{eq:action2}. To see that geometric fibers have class $\tilde \beta$, let $k$ be an algebraically closed field and let $(\specialP_k, \T_{\tilde\alpha}, \sigma, id)$ be a fiber of the tautological family over a $k$-point of $Z_{\tilde \beta}^0$. The fiber of \eqref{eq:Zbeta-invariance1} over the same point is $(\specialP_k, \T_{\tilde\alpha}, \wp \circ \sigma, id)$. Then as a quasimap to $Z\sslash T$, the class of $(\specialP_k, \T_{\tilde \alpha}, \wp \circ \sigma, id)$ is the homomorphism that sends $\L\in\Pic^T(Z)$ to 
\[
\deg_{\specialP} ((\wp\circ \sigma)^*(\T\times_T \L)).            
\]
Because $P_{\tilde \alpha}$ is a connected subgroup of $\Aut(\T_{\tilde \alpha}\times_T Z)$, there is a (piecewise linear) homotopy from the automorphism $\wp$ to the identity on $\T_{\tilde \alpha}\times_T Z$. In particular the images of $\sigma$ and $\wp\circ\sigma$ are rationally equivalent, hence the degree of $\T_{\tilde \alpha}\times_T \L$ along these two rational curves is the same.

\end{proof}

 Let $\beta = \rpic(\tilde \beta)$. We define
\begin{equation}\label{eq:defpsi}
\begin{aligned}
\psi_{\tilde \beta}: F^0_{\tilde \beta}(Z \sslash T) &\rightarrow F_{\beta}(Z\sslash G)\\
(C, \T, \sigma, \bx) &\mapsto (C, G \times_T \T, \sigma, \bx),
\end{aligned}\end{equation}
noting that $\sigma$ is a section of $
\T \times_T Z = (G \times_T \T)\times_G Z.$
A priori $\psi_{\tilde \beta}$ is a map to $QG_{\beta}(Z\sslash G)$; it is straightforward to check that it factors through $F_{\beta}(Z\sslash G)$. One uses the fact that isomorphisms of principal $T$-bundles induce isomorphisms of associated $G$-bundles.

\begin{lemma}\label{lem:psi-invariant}
The composition 
\begin{equation}\label{eq:psi-invariant1}
Z_{\tilde\beta}^0 \rightarrow F^0_{\tilde \beta}(Z\sslash T) \xrightarrow{\psi} F_{\beta}(Z\sslash G)
\end{equation}
is invariant under the action of $P_{\tilde\alpha}$ on $Z^0_{\tilde\beta}$.
\end{lemma}
\begin{proof}
As in the proof of Lemma \ref{lem:Zbeta-invariance}, we may replace $Z_{\tilde \beta}^0$ with a $G$-equivariant affine cover, and hence it is enough to show that \eqref{eq:psi-invariant1} is unchanged by precomposition with an arbitrary automorphism of $Z^0_{\tilde \beta}$ induced by $p: Z_{\tilde \beta}^0 \rightarrow P_{\tilde\alpha}$. The morphism \eqref{eq:psi-invariant1} is given by the family $(\specialP_{Z^0_{\tilde \beta}}, \T_{\tilde\alpha} \times_T G, \csigma, id)$ where $\csigma$ is the tautological section defined in \eqref{eq:ufam1}. The composition
\[
Z^0_{\tilde \beta} \xrightarrow{p} Z^0_{\tilde \beta} \xrightarrow{\text{\eqref{eq:psi-invariant1}}} F_{\beta}(Z\sslash G)
\]
is induced by the pullback of this family, which is precisely $(\specialP_{Z^0_{\tilde \beta}}, \T_{\tilde\alpha} \times_T G, \wp \circ \csigma, id)$ where $\wp \in \Aut(\T_{\tilde\alpha} \times_T G) \subset \Aut(\T_{\tilde\alpha} \times_T Z)$ is the automorphism defined by $p$ in Remark \ref{rmk:parabolic-auts}. This can be seen by checking that the square
\[
\begin{tikzcd}
\cZ \arrow[r, "{(x,\bu, y) \mapsto (px, \bu, y)}"] \arrow[d] &[50pt] \cZ \arrow[d] \\
\specialP_{Z^0_{\tilde \beta}} \arrow[u, "\wp \circ \csigma", bend left] \arrow[r, "{(x,\bu) \mapsto (px, \bu)}"] & \specialP_{Z^0_{\tilde \beta}} \arrow[u, bend right, "\csigma"']
\end{tikzcd}
\]
is a fibered square of fiber bundles with sections, where we interpret $\cZ$ and $\PP^1_{Z_{\tilde \beta}^0}$ using the GIT presentation in \eqref{eq:ufam1}, and we observe that in those coordinates we have $\wp \circ \csigma(x, \bu) = (x, \bu, \tau_{\tilde\alpha}(u)^{-1}px)$.
Since $\wp \in \Aut(\P_{\tilde\alpha})$, the families $(\specialP_{Z^0_{\tilde \beta}}, G \times_T \T_{\tilde\alpha}, \csigma, id)$ and $(\specialP_{Z^0_{\tilde \beta}}, G \times_T \T_{\tilde\alpha}, \wp \circ \csigma, id)$ are isomorphic quasimaps to $Z \sslash G$. 
\end{proof}

By Lemma \ref{lem:psi-invariant} we have an induced morphism
\begin{equation}\label{eq:induced}
Z_{\tilde\beta}^0/P_{\tilde\alpha} \rightarrow F_{\beta}(Z\sslash G).
\end{equation}

\begin{lemma}\label{lem:nonabelian}
The morphism \eqref{eq:induced} is a closed embedding.
\end{lemma}
\begin{proof}
There is a closed embedding $[Z_{\tilde \beta}^0/P_{\tilde\alpha}] \rightarrow [Z^s(G)/P_{\tilde\alpha}]$, but $[Z^s(G)/P_{\tilde\alpha}]$ is represented by a flag bundle on the variety $Z\sslash G$ that is proper over $\bS_G$, so $Z_{\tilde \beta}^0/P_{\tilde\alpha}$ is represented by a scheme that is proper over $\bS_G$. Since $F_\beta(Z\sslash G)$ is separated, the morphism \eqref{eq:induced} is proper, so to prove the lemma we need only show that it is a monomorphism; i.e., fully faithful.

Let $a_i:S \rightarrow Z^0_{\tilde \beta}$ be two morphisms from a scheme $S$. We show that if they induce isomorphic maps to $F_\beta(Z\sslash G)$, then they differ by an element of $P_{\tilde\alpha}$. The map to $F_\beta(Z\sslash G)$ induced by $a_i$ is given by the family $(\specialP_S, \T_{\tilde\alpha}\times_T G, \sigma_i, id)$ with \[
\sigma_i(x, \bu) = (x, \bu, \tau_{\tilde\alpha}(u)^{-1}a_i(x))
\]
using the GIT notation of \eqref{eq:ufam1}. Observe that $\sigma_{i, V} = a_i \circ pr_1$ where $pr_1:V_S = S\times V \rightarrow S$ is the projection. If these define isomorphic quasimaps to $Z\sslash G$ then by \eqref{eq:coordinateiso} there are maps $\phi_U: U_S \rightarrow G$ and $\phi_V:V_S \rightarrow G$ such that
\begin{equation}\label{eq:outerref}
\begin{gathered}
\phi_V\cdot(a_1\circ pr_1) = a_2\circ pr_1 \quad \quad \text{as maps}\;V_S \rightarrow Z\\
(\phi_V\circ \kappa) \tau_{\tilde\alpha} = \tau_{\tilde\alpha}\phi_U\quad \quad \text{as maps}\;U_S \setminus(S \times\{0\}) \rightarrow G.
\end{gathered}
\end{equation}
The first equation shows that $\phi_V$ is given by a composition $S\times V \xrightarrow{pr_1} S \xrightarrow{g} G$ for some $g \in G(S)$ sending $a_1$ to $a_2$. The second equation shows that $\tau_{\tilde\alpha}^{-1}g\tau_{\tilde\alpha}: S \times \CC^* \rightarrow G$ has an extension to $S\times \AA^1$ (the extension is $\phi_U: S\times U \rightarrow G$), so that $g \in P_{\tilde\alpha}(S)$. Hence the isomorphism defined by $\phi_U$ and $\phi_V$ is precisely the element of $\Aut(G \times_T \T_{\tilde\alpha})$ defined by $g \in P_{\tilde\alpha}(S)$ as in Remark \ref{rmk:parabolic-auts}.
\end{proof}

\section{Compute the $I$-function}\label{sec:compute}

\subsection{Weyl group action}

It is now our goal to show that the images of the closed embeddings \eqref{eq:induced} are always disjoint or equal, and to write $F_\beta(Z\sslash G)$ as a disjoint union of a certain collection of these images. 
Define
\[
 F^0_\beta(Z\sslash T) = \bigsqcup_{\tilde \beta \rightarrow \beta} F^0_{\tilde \beta}(Z\sslash T)
\]
and let $\psi: F^0_\beta(Z\sslash T) \rightarrow F_\beta(Z\sslash G)$ be defined to equal $\psi_{\tilde \beta}$ on $F_{\tilde \beta}^0(Z\sslash T)$. 

\begin{lemma}\label{lem:surjective}
The map $\psi:F^0_\beta(Z\sslash T) \rightarrow F_\beta(Z\sslash G)$ is surjective.
\end{lemma}
\begin{proof}Let $k$ be an algebraically closed field.
First we observe that if $(\specialP_k, \T_{\tilde\alpha}, \sigma, id)$ is a quasimap to $Z \sslash T$ with $\sigma_V$ constant, then this quasimap is fixed (necessarily with a unique basepoint at $[0:1]$). For, by \eqref{eq:coord1} we have 
\begin{equation}\label{eq:surj1}
\sigma_U(u) = \tau_{\tilde\alpha}(u)^{-1}\sigma_V
\end{equation}on $U_k\setminus\{0\}$. In fact \eqref{eq:surj1} determines $\sigma_U(u)$ on all of $U$,
and the underlying map $q:\specialP_k \rightarrow [Z/T]$ is given in homogeneous coordinates by $[u:v] \mapsto \tau_{\tilde \alpha}(u)^{-1}\sigma_V.$ One may check directly that this map is invariant under the $\CC^*$-action on $\specialP_k$.

Now let $(\specialP_k, \P_{\tilde\alpha}, \sigma, id)$ be a $k$-point of $F_\beta(Z\sslash G)$, for some $\tilde\alpha$ with $\rchi(\tilde\alpha) = \bd(\beta)$ (see Remark \ref{rmk:bundle}). We find an automorphism $\phi$ of $\P_{\tilde\alpha}$ sending $\sigma$ to a section $\rho$ with $\rho_V$ a constant function. By the above discussion, the quasimap $(\specialP_k, \T_{\tilde\alpha}, \rho, id)$ is a $k$-point of $F^0_{\tilde \beta}(Z\sslash T)$ and $(\specialP_k, \P_{\tilde\alpha}, \sigma, id)$ is in its essential image.

To define $\phi_V$, let $\iota: G \hookrightarrow Z^s(G)$ be the morphism defined by $\iota(g)=g\sigma_V(0)$. This is a closed embedding as follows: since $Z^s(G) \rightarrow Z\sslash G$ is a principal $G$-bundle, the map $G \times Z^s(G) \rightarrow Z^s(G) \times_{Z\sslash G} Z^s(G)$ is an isomorphism. On the other hand by \cite[Tag~02XE]{tag} there is a fiber square 
\[
\begin{tikzcd}
Z^s(G) \times_{Z\sslash G} Z^s(G) \arrow[r] \arrow[d] & Z\sslash G \arrow[d] \\
Z^s(G) \times Z^s(G) \arrow[r] & Z\sslash G \times Z\sslash G
\end{tikzcd}
\]
so since $Z\sslash G$ is a separated scheme, the composition $G \times Z^s(G) \rightarrow Z^s(G) \times Z^s(G)$ is a closed embedding.

We claim that $\sigma_V: V_k \rightarrow X^s(G)$ factors through the emedding $\iota$. This is because the quasimap is fixed, so $\sigma_V(V_k\setminus\{0\})$ is contained in a single $G$-orbit on $Z^s(G)$. Since $G$-orbits on $Z^s(G)$ are closed, $\sigma_V(0)$ must also be in this orbit. Now define $\phi_V=\iota^{-1}\sigma_V$, or in other words,
\begin{equation}\label{eq:surjective1}
\phi_V(v)\sigma_V(0) = \sigma_V(v).
\end{equation}

Now define
\begin{equation}\label{eq:surjective2}
\phi_U(u) = \tau_{\tilde \alpha}(u)^{-1}\phi_V(u^{-1})\tau_{\tilde \alpha}(u) \quad \quad \text{for}\; u \in U_k\setminus\{0\}.
\end{equation}
We show that $\phi_U$ extends to all of $U_k$. Embed $G$ as a closed subgroup in some $GL_n$. There is a commuting diagram
\[
\begin{tikzcd}
U_k\setminus\{0\} \arrow[r, "\phi_U"] & G \arrow[r, hook] & \AA^{n^2} \arrow[r, hook, "j_0"] \arrow[d, "\det"] &  \PP^{n^2}\arrow[d,dashed, "\det_{\PP}"] \arrow[r, hook] & {[\CC^{n^2+1}/\CC^*]} \arrow[d, "\det_{\PP}^{st}"]\\
&&\CC \arrow[r, "i_0", hook] &\PP^1 \arrow[r, hook] & {[\CC^2/\CC^*]}
\end{tikzcd}
\]
where $\AA^{n^2}$ is the ring of $n\times n$ matrices with coordinates $\{x_{ij}\}_{i, j=1}^n$, and all hooked arrows are embeddings with $j_0(x_{11}, \ldots, x_{nn}) = [x_{11}:\ldots x_{nn}: 1]$ and $i_0(x) = [x:1]$. The rational map $\det_{\PP}$ sends $[x_{11}:\ldots:x_{nn}:y]$ to $[\det([x_{ij}]): y^n]$ (it is defined on the image of $j_0$) and $\det_{\PP}^{st}$ is its natural extension to the stack quotients. Let $F$ denote the composition $U_k\setminus\{0\} \xrightarrow{\phi_U} G \hookrightarrow \PP^{n^2}$; then $F$ extends to a morphism $\tilde F: U_k \rightarrow \PP^{n^2}$ (see e.g. \cite[Thm~16.5.1]{vakil}). 

We show $\tilde F$ factors through $G$. First notice that the morphism $\det \circ \phi_V$ from $V_k\simeq\AA^1$ to $\CC^*$ must be constant, say equal to $d$. (One way to see this is that the corresponding ring map $\CC[x, x^{-1}] \rightarrow \CC[x]$ must send $x$ to a unit and hence to something in $k$.) So from the formula for $\phi_U$, the composition $\det_{\PP}^{st}\circ \tilde F$ is also the constant map to the point $[d:1]$. In particular $\tilde F$ factors through $j_0$ and even through $GL_n\subset\AA^{n^2}$. But $G$ is a closed subgroup of $GL_n$, so $\tilde F$ factors through $G$. 

By \eqref{eq:surjective2}, the morphisms $\phi_U$ and $\phi_V$ define an automorphism of $\P_{\tilde \alpha}$, and by \eqref{eq:surjective1} its inverse sends $\sigma$ to a section $\rho$ with $\rho_V=\sigma_V(0)$ a constant function.

\end{proof}

Define $ev_\bullet: F^0_\beta(Z\sslash T) \rightarrow Z^s(G)/T$ to equal $ev_\bullet$ on each component $F^0_{\tilde \beta}(Z\sslash T)$. Notice that $F^0_\beta(Z\sslash T)$ is a stack of maps to $[Z/T]$ and hence carries an action by the group scheme $W$ as in \eqref{eq:weyl2}. Under this action, $ev_\bullet$ is equivariant and $\psi$ is invariant. 

For $\tilde \alpha \in \Hom(\chi(T), \ZZ)$, let $W_{\tilde \alpha} = N_{L_{\tilde \alpha}}(T)/T$ be the Weyl group of $L_{\tilde \alpha}$, the unique Levi subgroup of $P_{\tilde \alpha}$ containing $T$. Recall that the group $W$ acts on $\Pic^T(Z)$ as in Section \ref{sec:wonpic}; this defines an action on $\Hom(\Pic^T(Z), \ZZ)$ analogous to \eqref{eq:dual_actions}. Observe also that since $W$ is finite, if $S$ is connected then there is a bijection between elements of $W(S)$ and $W(\Spec(k))$ for $k$ an algebraically closed field.

\begin{lemma}\label{lem:weyl_action}The action of the $W$ on $F^0_\beta(Z\sslash T)$ and $\rpic^{-1}(\beta) \subset \Hom(\Pic^T(Z), \ZZ)$ has the following properties.
\begin{enumerate}
\item If $(C, \T, \sigma, \bx)$ is an $S$-quasimap of class $\tilde \beta$ and $w \in W(S)$, then $w\cdot(C, \T, \sigma, \bx)$ has class $w\cdot \tilde \beta$. In particular the action of $W$ permutes the components $F_{\tilde \beta}^0(Z\sslash T)$  of $ F^0_\beta(Z\sslash T)$.
\item If $\tilde \beta_1, \tilde \beta_2 \in \Hom(\Pic^T(Z), \ZZ)$ are not in the same $W$-orbit, then the intersection of $\psi(F^0_{\tilde \beta_1}(Z\sslash T))$ and $\psi(F^0_{\tilde \beta_2}(Z\sslash T))$ is empty. 
\item If $\tilde \alpha = \bdtilde(\tilde \beta)$, then the stabilizer of $\tilde \beta \in \Hom(\Pic^T(Z), \ZZ)$ is $W_{\tilde \alpha}$.
\end{enumerate}
\end{lemma}
\begin{proof}
To prove (1), let $k$ be algebraically closed and let $(\specialP_k, \T, \sigma, id)$ be a $k$-quasimap in $F^0_{\tilde \beta}(Z\sslash T)$ and choose an equivariant line bundle $\L \in \Pic^T(Z)$. We use the description of $w \cdot (\specialP_k, \T, \sigma, id)$ in \eqref{eq:weyl2}. Then from \eqref{eq:wonpic1} we have a fiber square
\[
\begin{tikzcd}
\T \times w^{-1}(w^*\L) \arrow[r] \arrow[d]& w \T\times \L \arrow[d]\\
\T \times Z \arrow[r, "{(id, w\cdot)}"] &w\T\times Z
\end{tikzcd}
\]
where the horizontal maps are twisted-equivariant isomorphisms and vertical maps are $T$-equivariant. Hence we can quotient the square by $T$ to obtain a fiber square over $\varpi$ defined in \eqref{eq:wiso}. From this it follows that
\[
 \deg_{\specialP}(\varpi\circ\sigma)^*(w\T\times_T \L) = \deg_{\specialP}\sigma^*(\T\times_T w^{-1}(w^*\L)) = \deg_{\specialP}\sigma^*(\T\times_T (w^{-1}\cdot \L)),
\]
or in other words, the class of the quasimap $(\specialP_k, w\T, \varpi \circ \sigma, id)$ applied to $\L$ is $(w\cdot \tilde \beta)(\L).$

For (2),  for $i=1,2$ let $(\specialP_k, \T_{\tilde \alpha_i}, \sigma_i, id)$ be a $k$-quasimap of class $\tilde \beta_i$, and suppose these two  quasimaps have the same image in $F_\beta(Z\sslash G)$. Then in particular the associated $G$-bundles $\P_{\tilde \alpha_1}$ and $\P_{\tilde \alpha_2}$ are isomorphic, so by Theorem \ref{lem:bundle_abnonab} there is some $w \in N_G(T)$ such that $w\tilde \alpha_1 = \tilde \alpha_2$. Then the proof of Lemma \ref{lem:nonabelian} shows that there exists $p \in P_{\tilde \alpha_2}$ such that 
\[w \cdot (\specialP_k, \T_{\tilde \alpha_1}, \sigma_1, id) = (\specialP_k, \T_{\tilde \alpha_2}, \varpi \circ \sigma_1, id) = (\specialP_k, \T_{\tilde \alpha_2}, \wp\circ\sigma_2, id)\] (in particular, this argument did not require the two quasimaps to have the same class, just the same bundle type). Finally, as argued in the proof of Lemma \ref{lem:Zbeta-invariance}, the quasimaps $(\specialP_k, \T_{\tilde \alpha_2}, \wp\circ\sigma_2, id)$ and $(\specialP_k, \T_{\tilde \alpha_2}, \sigma_2, id)$ have the same class. So the class of $w\cdot(\specialP_k, \T_{\tilde \alpha_1}, \sigma_1, id)$ equals the class of $(\specialP_k, \T_{\tilde \alpha_2}, \sigma_2, id)$, which by part (1) implies $w \cdot \tilde \beta_1 = \tilde \beta_2$.

To prove (3), first note that by definition, $L_{\tilde \alpha}$ is the $G$-stabilizer of the cocharacter $\tau_{\tilde \alpha}$. Since the identification \eqref{eq:dual} of $\tilde \alpha$ and $\tau_{\tilde \alpha}$ is $W$-equivariant, we see that $W_{\tilde \alpha}$ is the stabilizer of $\tilde \alpha$. So the stabilizer of $\tilde \beta$ is a subgroup of $W_{\tilde \alpha}$.
Conversely, if $w \in N_{L_{\tilde \alpha}}(T)$ and $(\specialP_k, \T_{\tilde \alpha}, \sigma, id)$ is a quasimap of class $\tilde\beta$, we want to show that $(\specialP_k, \T_{w\tilde \alpha}, \varpi\circ\sigma, id)$ also has class $\tilde\beta$ (for then part (1) implies $w \cdot \tilde \beta = \tilde \beta$). 
Because $w$ is in the stabilizer of $\tilde \alpha$, the bundles $\T_{\tilde \alpha}$ and $\T_{w\tilde \alpha}$ are identically the same. In fact the morphism 
\[\varpi: \T_{\tilde \alpha} \times_T G \rightarrow \T_{w\tilde \alpha}\times_T G\] defined in \eqref{eq:wiso} is the same as the automorphism $\wp \in \Aut(\T_{\tilde \alpha}\times_T G)$ determined by $w$ as an element of $P_{\tilde \alpha}$. It was argued in the proof of argued in the proof of Lemma \ref{lem:Zbeta-invariance} that the class of $(\specialP_k, \T_{\tilde \alpha}, \wp \circ \sigma, id)$ is $\tilde \beta$. \end{proof}

Lemma \ref{lem:weyl_action} shows that the images 
\[F_{\tilde \beta}(Z\sslash G) := \psi(F^0_{\tilde \beta}(Z\sslash T))\] are either disjoint or identical subspaces of $F_\beta(Z\sslash G)$. Moreover, from Lemma \ref{lem:surjective} we know $\psi$ is surjective.
Let $\tilde \beta_i$ be elements of $\Hom(\Pic^T(Z),\ZZ)$ such that 
\begin{equation}\label{eq:components}F_\beta(Z\sslash G) =  \bigsqcup_i F_{\tilde \beta_i}(Z\sslash G),\end{equation}
i.e., the $\tilde \beta_i$ are elements of distinct Weyl orbits on $\Hom(\Pic^T(Z), \ZZ)$.
It follows from Lemmas \ref{lem:nonabelian} and \ref{lem:finite} that \eqref{eq:components} is a decomposition of $F_\beta(Z\sslash G)$ as a disjoint union of open and closed subschemes.

\subsection{Relate the perfect obstruction theories}
The main goal of this section is to relate the perfect obstruction theory of $F^0_{\tilde \beta}(Z\sslash T)$ to the pullback of the perfect obstruction theory of $F_{\tilde \beta}(Z\sslash G)$ under $\psi_{\tilde \beta}$. Let 
\[
\EE_G := \EE_{QG_\beta(Z\sslash G)} \quad \quad \EE_T := \EE_{QG_{\tilde \beta}(Z\sslash T)}
\]
denote the absolute perfect obstruction theories defined in \eqref{eq:pot}. In what follows, if $A$ (resp. $B$) is a complex on $QG_\beta(Z\sslash G)$ (resp. $QG_{\tilde \beta}(Z\sslash T)$), we will use the notation
 \[
 A|_F := A |_{F_{\tilde \beta}(Z\sslash G)} \quad \quad (\text{resp.}\;\, B|_F := B |_{F_{\tilde \beta}^0(Z\sslash T)})
 \]
 for the restricted complex whenever the intended class $\tilde \beta$ is clear. In particular, we have
\[
\EE_G|_F := \EE_G|_{F_{\tilde\beta}(Z\sslash G)} \quad \quad \EE_T|_F := \EE_T|_{F_{\tilde\beta}^0(Z\sslash T)}.
\]

To relate the obstruction theories, we use the tower of morphisms 
\[\specialP \times [X/T]\xrightarrow{\psi} \specialP \times[X/G] \rightarrow \specialP \rightarrow pt\]
which leads to a morphism of moduli of sections 
\[
\psi: \Sec{pt}{\specialP\times[Z/T]}{\specialP}\rightarrow \Sec{pt}{\specialP\times[Z/G]}{\specialP}.\] 
Define $QG^0_{ \beta}(Z\sslash T)$ and the map $\psi^0$ to be the fiber product
\[
\begin{tikzcd}
QG^0_{\beta}(Z\sslash T) \arrow[r] \arrow[d, "\psi^0"] & \Sec{pt}{\specialP\times{[Z/T]}}{\specialP} \arrow[d, "\psi"] \\
QG_{\beta}(Z\sslash G) \arrow[r] & \Sec{pt}{\specialP\times{[Z/G]}}{\specialP}
\end{tikzcd}
\]
Notice that $QG^0_\beta(Z\sslash T)$ is an open substack of the (finite) disjoint union of moduli spaces $QG_{\tilde \beta}(Z\sslash T)$ with $\tilde \beta$ mapping to $\beta$. 

\begin{lemma}\label{lem:dt}
In the derived category of $F_{\tilde \beta}^0(Z\sslash T)$, there is a morphism of distinguished triangles
\begin{equation}\label{eq:dt}
\begin{tikzcd}
\psi^*_{\tilde \beta}( \EE_G|_F) \arrow[r]\arrow[d] & \EE_T|_F \arrow[r] \arrow[d]& (R \pi_* n_F^* \TT_{\psi})^\vee \arrow[d]\arrow[r] & {}\\
\psi^*_{\tilde \beta}(\LL_{QG_\beta(Z\sslash G)}|_F) \arrow[r] & \LL_{QG_{\tilde \beta}(Z\sslash T)}|_F \arrow[r] &  \LL_{\psi^{0}}|_F \arrow[r] & {}
\end{tikzcd}
\end{equation}
where $\psi$ is the canonical map $[Z/T]\rightarrow[Z/G]$ and $n_F$ is the restriction to $\specialP \times F^0_{\tilde \beta}(Z\sslash T)$ of the universal map defined in Section \ref{sec:pot}.
\end{lemma}
\begin{proof}

On $QG^0_{\beta}(Z\sslash T)$ we have the following morphism of distinguished triangles, where the middle and left vertical arrows are the absolute perfect obstruction theories of \eqref{eq:pot} (see \cite[Lem~A.2.3]{cjw}).
\begin{equation}\label{eq:dt1}
\begin{tikzcd}
(\psi^0)^* \EE_G \arrow[r] \arrow[d, "(\psi^0)^*\phi_G"] & \EE_T \arrow[r] \arrow[d, "\phi_T"] & (R \pi_* n^* \TT_{\psi})^\vee \arrow[r]\arrow[d] & {}\\
(\psi^0)^* \LL_{QG_{\tilde \beta}(Z\sslash G)} \arrow[r] & \LL_{QG_{\tilde \beta}(Z\sslash T)} \arrow[r] &\LL_{\psi^0} \arrow[r] & {}
\end{tikzcd}
\end{equation}
Now restrict this diagram to $F^0_{\tilde \beta}(Z\sslash T)$ and use that $R\pi_*$ commutes with restriction to the fixed locus by \cite[Cor~4.13]{HR17}.

\end{proof}

We use Lemma \ref{lem:dt} to relate the virtual and Euler classes appearing in the definition \eqref{eq:Ifunc} of the $I$-function. We recall the definitions of these classes. According to \cite[Sec~3]{CKL17}, the composition
\[
\EE_G |_F^{\fix} \xrightarrow{\phi|_F^\fix} \LL_{QG_\beta(Z\sslash G)}|_F^{\fix} \rightarrow \LL_{F_{\tilde \beta}(Z\sslash G)}
\]
is a perfect obstruction theory on $F_{\tilde \beta}(Z\sslash G)$. The virtual class $[F_{\tilde \beta}(Z\sslash G)]^{\vir}$ in \eqref{eq:Ifunc} is the one defined by this perfect obstruction theory. By definition we have 
\begin{equation}\label{eq:virtnorm}
N^{\vir}_{F_{\tilde \beta}(Z\sslash G)} := (\EE_G |_F^{\mov})^\vee.
\end{equation}
We note that the complex \eqref{eq:virtnorm} has a global resolution by vector bundles. (This follows from \cite[Tag~0F8E]{tag} and the fact that $F_{\tilde \beta}(Z\sslash G)$ has the resolution property, see \cite[Thm~2.1]{To04}.) 
Thus we may define its equivariant Euler class as in \cite[Def~3.3]{CKL17}; i.e., if $N^\bullet$ is a finite complex of vector bundles with $e_{\CC^*}(N^i)$ invertible for $i$ odd, then 
\begin{equation}\label{eq:defeuler}
e_{\CC^*}(N) = \prod_i e_{\CC^*}(N^i)^{(-1)^i}.
\end{equation}
One may show as in \cite[Tag~0ESZ]{tag} that the definition \eqref{eq:defeuler} depends only on the complex represented by $N$ and not on the choice of resolution. We make a simple observation about this definition that is useful in our computation. 
\begin{lemma}\label{lem:compute-euler}
If moreover the cohomology sheaves $H^i(N)$ are locally free, then
$e(N) = \prod_i e(H^i(E))^{(-1)^i}$.
\end{lemma}
\begin{proof}
Let $N$ be the complex $\ldots \rightarrow N^n \xrightarrow{d_n}  N^{n+1} \rightarrow \ldots$ with locally free cohomology sheaves. For each $n$, we have short exact sequences
\begin{gather*}
0 \rightarrow \ker(d_n) \rightarrow N^n \rightarrow \Im(d_n) \rightarrow 0\\
0 \rightarrow \Im(d_{n-1}) \rightarrow \ker(d_n) \rightarrow H^n(N) \rightarrow 0.
\end{gather*}
The result now follows from a routine computation using additivity of the Euler class and the fact that $\ker(d_n)$ and $\Im(d_n)$ are locally free by \cite[Tag~0F8J]{tag}.

\end{proof}


\begin{corollary}\label{cor:pullback}
We have the following relationships on $F_{\tilde \beta}^0(X\sslash T)$:
\begin{align}
\psi_{\tilde \beta}^*[F_{\tilde \beta}(Z\sslash G)]^{\vir}&=[F^0_{\tilde \beta}(Z\sslash T)]^{\vir} \label{eq:relatevir}\\
\psi_{\tilde \beta}^*e_{\CC^*}(N^{\vir}_{F_{\tilde \beta(Z\sslash G)}})=e_{\CC^*}(N^{\vir}_{F^0_{\tilde \beta(Z\sslash T)}})& \frac{\prod_{\tilde\beta(\rho)<0}\prod_{k=\tilde\beta(\rho)+1}^{-1}(c_1(\L_\rho)+kz)}{\prod_{\tilde\beta(\rho)\geq 0} \prod_{k=1}^{\tilde \beta(\rho)}(c_1(\L_\rho)+kz)}.\label{eq:relateeuler}
\end{align}
Here, $\rho$ ranges over roots of $G$ with respect to $T$.
\end{corollary}
\begin{proof}
For \eqref{eq:relatevir}, modify \eqref{eq:dt} by applying the ``fix'' functor, and then use the commuting square
\[
\begin{tikzcd}
F^0_{\tilde \beta}(Z\sslash T)  \arrow[r] \arrow[d, "\psi_{\tilde \beta}"] & QG_{\beta}^0(Z\sslash T) \arrow[d, "\psi^0"] \\
F_{\tilde \beta}(Z\sslash G) \arrow[r]& QG_{\beta}(Z\sslash G)
\end{tikzcd}
\]
and functoriality of the cotangent complex to map the bottom row of the fixed part of \eqref{eq:dt} to the canonical distinguished triangle for $\psi_{\tilde \beta}$. The resulting morphism of distinguished triangles
\[
\begin{tikzcd}
\psi^*_{\tilde \beta}( \EE_G|^\fix_F) \arrow[r]\arrow[d] & \EE_T|^\fix_F \arrow[r] \arrow[d]& ((R \pi_* n_F^* \TT_{\psi})^\fix)^\vee \arrow[d, "\phi_{\psi}"]\arrow[r] & {}\\
\psi^*_{\tilde \beta}(\LL_{F_{\tilde \beta}(Z\sslash G)}) \arrow[r] & \LL_{F^0_{\tilde \beta}(Z\sslash T)} \arrow[r] &  \LL_{\psi_{\tilde \beta}} \arrow[r] & {}
\end{tikzcd}
\]
is exactly the diagram for checking that we can define virtual pullback along $\psi_{\tilde \beta}$ as in \cite{manolache}. Observe that $\psi_{\tilde \beta}$ is smooth (being the projection morphism of a flag bundle). We claim that the arrow $\phi_{\psi}$ is a quasi-isomorphism; granting this, the diagram above implies that virtual pullback along $\psi_{\tilde \beta}$ is defined and agrees with the usual flat pullback \cite[Rmk~3.10]{manolache}. By \cite[Cor~4.9]{manolache}, we get \eqref{eq:relatevir}.

To show that $\phi_{\psi}$ is a quasi-isomorphism, it suffices to show that it induces an isomorphism of cohomology sheaves of degree -1 (since a standard diagram chase shows that $\phi_{\psi}$ is an obstruction theory). Because $\LL_{\psi_{\tilde\beta}}$ is represented by a vector bundle in degree 0, it suffices to show that $(R^1\pi_*n_F^*\TT_{\psi})^{\fix}=0.$ By Nakayama's lemma it suffices to check that fibers at closed points vanish. If $q = (\specialP \xrightarrow{n} [Z/T])$ is a closed point of $F^0_{\tilde \beta}(X\sslash T)$ then $(R^1\pi_*n_F^*\TT_{\psi})|_q = H^1(\specialP, n^*\TT_{\psi})$. Since the $\CC^*$-linearization on $n^*\TT_{\psi}$ is trivial, and $H^1(\specialP, n^*\TT_{\psi})$ has a basis of monomials in $u,v$ where each variable has degree at most -1, we see that this representation has no fixed part.

To compute \eqref{eq:relateeuler}, observe from the universal family \eqref{eq:ufam1} that we may write the universal curve $F^0_{\tilde \beta}(Z\sslash T) \times \specialP$ as the quotient
$(Z^0_{\tilde \beta} \times (\CC^2\setminus \{0\}) / (T \times \CC^*)$
and that with this presentation, the vector bundle $n_F^*\TT_{\psi}$ on $F^0_{\tilde \beta}(Z\sslash T)\times \specialP$ is induced from a topologically trivial bundle on $Z^0_{\tilde \beta}\times (\CC^2\setminus\{0\})$. This trivial bundle has fiber equal to the subspace of the lie algebra $\g$ of $G$ with nontrivial weights, viewed as a $T \times \CC^*$ representation via the homomorphism
\[
T \times \CC^* \xrightarrow{(t, s) \rightarrow t\tau_{\tilde \alpha}(s)^{-1}} T
\]
and the adjoint representation of $T$ on $\g$. In other words, $n_F^*\TT_{\psi}$ splits as a direct sum of line bundles
\[
n_F^*\TT_{\psi} = \oplus_{\rho} \pi^*\L_\rho\otimes\OO_{\specialP\times F_{\tilde \beta}(Z\sslash T)}(\tilde \beta(\rho))
\]
where the sum ranges over the roots of $\g$ relative to $T$. By the projection formula and flat base change, we have
\[
R^i\pi_*(n_F^*\TT_\psi) = \oplus_\rho \L_\rho \otimes R^i\pi_*(\OO_{\specialP}(\tilde \beta(\rho)),
\]
and in particular these sheaves are locally free.
Now we apply \cite[Tag~0F8G,~0F9F]{tag} and Lemma \ref{lem:compute-euler} to the dual of the top row of \eqref{eq:dt}, recalling the definition \eqref{eq:virtnorm}. We get
\begin{equation}\label{eq:compute5}
\psi_{\tilde \beta}^*e_{\CC^*}(N^{vir}_{F_{\tilde \beta}(Z\sslash G)}) = e_{\CC^*}(N^{vir}_{F^0_{\tilde \beta}(Z\sslash T)})  \frac{e_{\CC^*}((R^1\pi_*n_F^*\TT_\psi)^\mov)}{e_{\CC^*}((R^0\pi_*n_F^*\TT_\psi)^\mov)}.
\end{equation}
If $\tilde \beta(\rho)$ is nonnegative, then $R^1\pi_*(\OO_{\specialP}(\tilde \beta(\rho)))$ vanishes, but $R^0\pi_*(\OO_{\specialP}(\tilde \beta(\rho)))$ is nonzero on a closed fiber of $\pi$, and a basis is given by the monomials $u^{\tilde \beta(\rho)}, u^{\tilde \beta(\rho)-1}v, u^{\tilde \beta(\rho)-2}v^2, \ldots, v^{\tilde \beta(\rho)}$ which have $\CC^*$-weights $0, 1, 2, \ldots, \tilde \beta(\rho)$, respectively. Hence the Euler class of the moving part of the corresponding summand of $R^0\pi_*n_F^*\TT_\psi$ is $\prod_{k=1}^{\tilde \beta(\rho)}(c_1(\L_\rho)+kz)$.

If $\tilde \beta(\rho)$ is less than -1, then $R^0\pi_*(\OO_{\specialP}(\tilde \beta(\rho))$ vanishes but $R^1\pi_*(\OO_{\specialP}(\tilde \beta(\rho))$ is nonzero on a closed fiber of $\pi$, and a basis is given by monomials $u^{-1}v^{\tilde \beta(\rho)+1}, u^{-2}v^{\tilde \beta(\rho)+2}, \ldots, u^{\tilde \beta(\rho)+1}v^{-1}.$ Hence the Euler class of the moving part of the corresponding summand of $R^1\pi_*n_F^*\TT_\psi$ is $\prod_{k=\tilde\beta(\rho)+1}^{-1}(c_1(\L_\rho)+kz).$
\end{proof}
\subsection{Proof of the main theorem}\label{sec:proof}

The following lemma, a restatement of \cite[Prop~2.1]{brion}, lets us navigate around the bottom left triangle of \eqref{eq:big_diagram}. 

\begin{lemma}\label{lem:brion}
For any $\delta \in A_*(Z^s(G)/P_{\tilde \alpha})$, we have
\begin{equation}\label{eq:point}
g^*f_*\delta = \sum_{w\in W/W_{\tilde \alpha}} w^*\left[ \frac{p^*\delta}{\prod_{\rho\in R^-_{\tilde \alpha}} c_1(\L_{\rho})}\right]
\end{equation}
where $R^-_{\tilde \alpha}$ is the set of roots of $G$ whose inner product with the dual character $\tilde \alpha$ is negative.
\end{lemma}
\begin{proof}
We reduce this statement to the one in \cite[Prop~2.1]{brion}. Using the dynamic method, one may obtain a Borel subgroup $B$ of $G$, contained in $P_{\tilde \alpha}$, equal to $P_\mu$ for some cocharacter $\mu$ that is positive on any root where $\tau_{\tilde \alpha}^{-1}$ is positive (see e.g. \cite[45]{confeng}, noting that our definition of $P_{\tilde \alpha}$ is dual to the one in that reference). So the opposite roots of this Borel, minus the roots of $L_{\tilde \alpha}$, are precisely those roots where $\tau_{\tilde \alpha}^{-1}$ is negative. Recalling the relationship \eqref{eq:dual}, we see that this is the set $R_{\tilde \alpha}^-$.

The statement of \cite[Prop~2.1]{brion} is for classes in $A_*(Z^s(G)/P_{\tilde \alpha})_\QQ$ that are in the image of 
\[c^{W_{\tilde \alpha}}: Sym(\chi(T)_\QQ)^{W_{\tilde \alpha}} \rightarrow A_*(Z^s(G)/P_{\tilde \alpha})_\QQ,\]
a morphism defined in \cite[47]{brion}. We claim that when the classes in the image of $c^{W_{\tilde \alpha}}$ are restricted to any fiber of $Z^s(G)/P_{\tilde \alpha} \rightarrow Z^s(G)/G$, they generate the Chow group of that fiber, so that by the Leray-Hirsch theorem \cite[Lem~6]{EG97} it suffices to show \eqref{eq:point} for $\delta$ in the image of $c^{W_{\tilde \alpha}}$. To prove the claim, let $z \in Z^s(G)/G$. We have a commuting diagram
\[
\begin{tikzcd}
Sym(\chi(T)_\QQ) \arrow[r, "c"] & A_*(Z^s(G)/B) \arrow[r, "\iota_z^*"] & A_*(G/T) \\
Sym(\chi(T)_\QQ)^{W_{\tilde \alpha}} \arrow[u, hookrightarrow] \arrow[r, "c^{W_{\tilde \alpha}}"] & A_*(Z^s(G)/P_{\tilde \alpha}) \arrow[u] \arrow[r, "\iota_z^*"] & A_*(G/P_{\tilde \alpha}) \arrow[u, "F"]
\end{tikzcd}
\]
where the maps labeled $\iota_z^*$ are restrictions to fibers over $z$, vertical maps of Chow groups are pullbacks, the left square commutes by definition of $c^{W_{\tilde \alpha}}$, and we have used that the pullback $A_*(G/B) \rightarrow A_*(G/T)$ is an isomorphism \cite[(2.6)]{ellingsrud}. The morphism $c$ is the characteristic homomorphism, and by \cite[(1.3)]{ellingsrud} the composition of the top arrows is surjective. By Proposition \ref{prop:chow}, the map $F$ is injective. Now a diagram chase shows that the composition of the bottom arrows is surjective, as desired.

Now for $\delta = c^{W_{\tilde \alpha}}(\xi)$, the result \cite[Prop~2.1]{brion} tells us
\[
g^*f_*\delta = p^*f^*f_*c^{W_{\tilde \alpha}}(\xi) = p^*c^{W_{\tilde \alpha}}\left(\sum_{w \in W/W_{\tilde \alpha}}w\cdot(\xi/\prod_{\rho \in R^-_{\tilde \alpha}}\rho) \right).
\]
It follows from the definition of $c^{W_{\tilde \alpha}}$ that the composition $p^*c^{W_{\tilde \alpha}}: Sym(\chi(T)_\QQ)^{W_{\tilde \alpha}} \rightarrow A_*(Z^s(G)/T)_{\QQ}$ sends a character $\xi$ to $c_1(\L_\xi)$. By \eqref{eq:wonpic2} this composition is $W$-equivariant, so \eqref{eq:point} follows.

\end{proof}

Let $\tilde \alpha_i = \bdtilde(\tilde \beta_i)$. Turning to formula \eqref{eq:Ifunc} for $I_\beta^{Z\sslash G}(z)$, we first write it as a sum of pushforwards from $F_{\tilde \beta_i}(Z\sslash G)$ using \eqref{eq:components}. We use Proposition \ref{prop:diagram} to identify the evaluation map on each component, and then apply Lemma \ref{lem:brion}, obtaining
\begin{equation}\label{eq:start}
g^*I^{Z\sslash G}_\beta = \sum_{\tilde \beta_i} \sum_{w \in W/W_{\bdtilde(\tilde\beta_i)}} w^*\left[ \frac{p^*i_*([F_{\tilde \beta_i}(Z\sslash G)]^{vir}e_{\CC^*}(N_{F_{\tilde \beta_i}(Z\sslash G)}^{vir})^{-1})}{\prod_{\rho\in R^-_{\tilde \alpha_i}} c_1(\L_{\rho})}\right].
\end{equation}

Let us simplify the numerator of a summand of \eqref{eq:start}. From Lemma \ref{lem:weyl_action} and the equivariance of $ev_{\bullet}$, we have a commuting diagram
\begin{equation}\label{eq:twosquares}
\begin{tikzcd}
F^0_{w^{-1}\tilde \beta_i}(Z\sslash T) \arrow[r, "w"] \arrow[d,hook, "ev_\bullet"]& F^0_{\tilde \beta_i}(Z\sslash T) \arrow[r, "\psi_{\tilde \beta_i}"] \arrow[d,hook, "ev_\bullet"]& F_{\tilde \beta_i}(Z\sslash G) \arrow[d,hook, "i"]\\
Z^s(G)/T \arrow[r,"w"] & Z^s(G)/T \arrow[r,"p_{\tilde \alpha_i}"] & Z^s(G)/P_{\tilde \alpha_i}
\end{tikzcd}
\end{equation}
The square on the left is fibered because $w$ is an isomorphism and the square on the right is fibered by Proposition \ref{prop:diagram}, so the outer square is fibered.
Because $w$ and $p_{\tilde \alpha_i}$ are flat, by \cite[Prop~1.7]{Fu98} we have $w^*p_{\tilde \alpha_i}^*i_* = (ev_\bullet)_*w^*\psi_{\tilde \beta_i}^*$, so that the numerator of a summand in \eqref{eq:start} is
\begin{equation}\label{eq:num}
(ev_\bullet)_*\psi_{w^{-1}\tilde \beta_i}^*\big([F_{\tilde \beta_i}(Z\sslash G)]^{vir}e_{\CC^*}(N_{F_{\tilde \beta_i}(Z\sslash G)}^{vir})^{-1}\big)
\end{equation}
where we have also used that $\psi$ (defined on $ F^0_\beta(Z\sslash T)$) is equivariant.

Now let us compute the denominator of a summand of \eqref{eq:start}. We get 
\begin{equation}\label{eq:denom}
w^*\prod_{\rho \in R^-_{\tilde \alpha_i}} c_1(\L_\rho) = \prod_{\rho \in R^-_{\tilde \alpha_i}} c_1(\L_{w^{-1}\cdot\rho}) = \prod_{\rho \in R^-_{w^{-1}\cdot\tilde \alpha_i}} c_1(\L_\rho).
\end{equation}
The first equality uses \eqref{eq:wonpic2} and the second follows from the fact that the natural pairing between $\chi(T)$ and $\Hom(\chi(T),\ZZ)$ is invariant.

Finally we apply equations \eqref{eq:num} and \eqref{eq:denom} and use Lemma \ref{lem:weyl_action} to combine the double sum in \eqref{eq:start} into a single sum, obtaining
\begin{equation}\label{eq:mainproof1}
g^*I_\beta^{Z\sslash G} = \sum_{\tilde \beta \mapsto \beta} \frac{(ev_\bullet)_*\psi_{\tilde \beta}^*([F_{\tilde \beta}(Z\sslash G)]^{vir}e_{\CC^*}(N_{F_{\tilde \beta}(Z\sslash G)}^{vir})^{-1}))}{\prod_{\rho \in R^-_{\bdtilde(\tilde\beta)}}c_1(\L_\rho)}.
\end{equation}
We can compute the pullbacks in the numerator with Corollary \ref{cor:pullback}. Finally, applying the projection formula and recalling that $R^-_{\bdtilde(\tilde \beta)}$ is just the set of roots with $\bdtilde(\tilde \beta)(\rho)=\tilde \beta(\rho)<0$, we recover Theorem \ref{thm:main}.

\section{Extensions and applications}

\subsection{Equivariant $I$-functions}\label{sec:equivariant} Let $S$ be a torus and suppose that we have an action of $S \times G$ on $Z$ extending the action of $G = \{1\} \times G$ on $Z$. In other words, $S$ acts on $Z$ and this action commutes with the action of $G$. Then $S$ acts on $[Z/G]$ and $[Z/T]$ (see \cite[Rmk~2.4]{romagny}) and this defines actions on $QG_\beta(X\sslash G)$ and $QG_{\tilde \beta}(X\sslash T)$ and their universal families by \cite[Sec~A.3]{cjw}, viewing them as substacks of the moduli of sections as in Section \ref{sec:qmapmoduli}. Moreover the perfect obstruction theories $E_{QG}$ in \eqref{eq:pot} are canonically $S$-equivariant as in \cite[Sec~A.3]{cjw}.

Because the actions of $S$ and $\CC^*$ on $\PP^1 \times [Z/G]$ commute, the $\CC^*$-fixed locus $F_{ \beta}(Z\sslash G)$ is invariant under the action of $S$ and the $\CC^*$-fixed and moving parts of the perfect obstruction theory $E_{QG_\beta}$ are also $S$-equivariant. Finally the map $ev_{\bullet}$ is $S$-equivariant since the universal family on $QG_\beta(X\sslash G)$ is. These statements also hold for $T$ in place of $G$. Since the spaces $F_\beta(Z\sslash G)$ are schemes, we can use the equivariant intersection theory of \cite{EG98} to define $[F_\beta(Z\sslash G)]^{S, \vir}$ in $A_*^S(F)$. The class $e_{S\times \CC^*}(N_{F_\beta(Z\sslash G)}^\vir)$ is defined as in \cite[Tag~0F9E]{tag} but with the Euler classes replaced by their $S \times \CC^*$-equivariant counterparts. Hence we can define the $S$-equivariant $I$-function via the same formulas \eqref{eq:Ifunc}, but with all objects replaced by their $S$-equivariant counterparts.

For $\rho \in \chi(T)$, let $L_{\rho}$ be the $S$-equivariant line bundle on $X^s(T)/T$ given by 
\begin{equation}\label{eq:equivariant1}
L_\rho = X^s(T)\times_T \CC_\rho
\end{equation}
where $\CC_\rho$ is the $S \times T$-equivariant representation where $S$ acts trivially and $T$ acts with character $\rho$. 

\begin{corollary}\label{cor:equivariant}
The $S$-equivariant $I$-functions of $Z\sslash G$ and $Z\sslash T$ satisfy the equation \eqref{eq:main}, with $I^{S, Z\sslash G}_\beta(z)$ and $I^{S, Z\sslash T}_\beta(z)$ in place of $I^{ Z\sslash G}_\beta(z)$ and $I^{ Z\sslash T}_\beta(z)$.
\end{corollary}

\begin{proof}
First note that Proposition \ref{prop:chow} and Lemma \ref{lem:brion} also hold $S$-equivariantly (in Lemma \ref{lem:brion}, the line bundles $c_1(\L_\rho)$ are $S$-equivariant as in \eqref{eq:equivariant1} and we take the $S$-equivariant first Chern class). The same proofs work after replacing $Z^s(G)$ with $Z^s(G)\times_S U$, where $U \rightarrow U/S$ is an appropriate approximation of the universal $S$-bundle (definition as in \cite[Sec~2.2]{EG98}).


Now the computation in Section \ref{sec:compute} proceeds as follows. The diagram \eqref{eq:dt} is equivariant; i.e., it is isomorphic to the pullback of a morphism of distinguished triangles on $[F^0_{\tilde \beta}(Z\sslash T)/S]$. This is true because the diagram \eqref{eq:dt1} is equivariant by \cite[Lem~A.3.3]{cjw}. Next, to compute the equivariant Euler class in Corollary \ref{cor:pullback}, note that since $S$ commutes with $G$ its action on the lie algebra $\g$ is trivial. The remainder of the proof is the same as in the non-equivariant case.
\end{proof}

\subsection{Twisted $I$-functions}
Let $S$ be a torus and suppose we have an action of $S \times G$ on $Z$ as in Section \ref{sec:equivariant}. Furthermore,
let $R=\CC^*$ act trivially on $Z$ with equivariant parameter $\mu$; note this induces the trivial action on $F_\beta(Z\sslash G)$ as a moduli space of maps.
Let $E$ be a $S\times G$-equivariant vector bundle on $Z$, and let $\CC_\mu$ be the $R$-equivariant vector bundle on $Z$ that is topologically trivial and has its $R$-action given by scaling fibers. Let $E_G$ denote the $S\times R$-equivariant vector bundle on $[Z/G]$ corresponding to $E\otimes \CC_\mu$. Recall that $\CC^*$ acts on $\specialP$ via \eqref{eq:action1} and hence on $F_\beta(Z\sslash G) \times \PP^1$, and that the universal map $n:F_\beta(Z\sslash G) \times \PP^1 \rightarrow [Z/G]$ is invariant with respect to this action. So $n^*E_G$ is naturally $S\times R \times \CC^*$-equivariant.
We assume that the complex $R\pi_*n^*E_G$ has a $S\times R \times \CC^*$-equivariant global resolution by vector bundles; i.e., it is an element of the rational Grothendieck group
\[
K^\circ_{S\times R \times \CC^*}(F_\beta(Z\sslash G) = K^\circ_{S\times\CC^*}(F_\beta)\otimes \QQ[\mu, \mu^{-1}]
\]
of $S\times R \times\CC^*$-equivariant vector bundles on $F_\beta(Z\sslash G)$. This assumption holds, for example, if $R^1\pi_*n^*E_G$ is zero and $R^0\pi_*n^*E_G$ is a vector bundle (see also \cite[Sec~6.2]{stable_qmaps}). 

Fix an invertible multiplicative characteristic class $\bc$ defining a group homomorphism
\[
\bc: K^\circ_{S\times R \times \CC^*}(F_\beta(Z\sslash G)) \rightarrow (H^*_{S\times R \times\CC^*}(F_\beta(Z\sslash G), \QQ))^\times
\]
to the group of units in $H^*_{S \times R \times \CC^*}(F_\beta(Z\sslash G), \QQ)$. A priori, $\bc$ may be defined only for vector bundles; its invertibility means its definition extends to elements of $K^\circ$. Let $\underline{E}_G$ denote the $S\times R$-equivariant vector bundle on $Z\sslash G$ induced by $E\otimes \CC_{\mu}$. Now we define the $S$-equivariant, $\bc(E)$-twisted $I$-function to be
\[
I^{Z\sslash G,\,S,\, \bc(E)}(z) = 1+ \sum_{\beta\neq 0}q^\beta I^{Z\sslash G, \,S,\,\bc(E)}_\beta(z)
\]
where
\begin{equation}\label{eq:twisted1}
I^{Z\sslash G, \,S,\,\bc(E)}_\beta(z) = \bc(\underline{E}_G)^{-1}(ev_{\bullet})_*\left(\frac{[F_\beta(Z\sslash G)]^{S\times R,\vir}\cap \bc(R\pi_*n^*E_G)}{e_{S\times R \times\CC^*}(N^{vir}_{F_{\beta}(Z\sslash G)})}\right).
\end{equation}
(see \cite[(7.2.3)]{wcgis0}). Note that the torus $R$ is omitted from the superscripts in the $I$-function notation.

For the abelianization theorem, observe that $E$ is naturally a $T$-equivariant vector bundle on $Z$, so we can also define the $\bc(E)$-twisted $I$-function of $Z\sslash T$.

\begin{corollary}\label{cor:twisted}
If the class $\bc$ is functorial with respect to pullback, then Theorem \eqref{thm:main} holds with $I^{Z\sslash G, \,S,\,\bc(E)}_\beta(z)$ and $I^{Z\sslash T, \,S,\,\bc(E)}_\beta(z)$ in place of $I^{ Z\sslash G}_\beta(z)$ and $I^{ Z\sslash T}_\beta(z)$.
\end{corollary}
\begin{proof}
To complete the computation in Section \ref{sec:proof}, first note that 
\[g^*\bc(\underline{E}_G)^{-1} = \bc(g^*\underline{E}_G)^{-1} = \bc(\underline{E}_T)^{-1}.\] 
The remainder of the computation is the same until the last line when we replace the numerator in the right-hand side of \eqref{eq:mainproof1} with
\[
{(ev_\bullet)_*\psi_{\tilde \beta}^*([F_{\tilde \beta}(Z\sslash G)]^{vir}\cap e_{S\times R\times\CC^*}(N_{F_{\tilde \beta}(Z\sslash G)}^{vir})^{-1} \cap\bc(R\pi_*n^*E_G))}.
\]
By functoriality of $\bc$, the term $\psi_{\tilde \beta}^*(\bc(R\pi_*n^*E_G))$ is equal to
\[
\bc(\psi_{\tilde \beta}^*R\pi_*n^*E_G) = \bc(R\pi_*n^*\psi^*E_G)
\]
where $\psi$ is the natural map from $[Z/T]$ to $[Z/G]$. The bundle $\psi^*E_G$ is just $E_T$.
\end{proof}

\begin{remark}\label{rmk:twist}The standard application of twisted invariants is to choose $E$ that satisfies $R^1\pi_*n^*E_G=0$ and set $\bc$ to be the $R$-equivariant Euler class $e_S$. Then the non-equivariant limit of \eqref{eq:twisted1} exists---i.e., one can set $\mu=0$. This non-equivariant limit is the definition of the twisted $I$-function in \cite[Sec~7.2]{wcgis0}. Taking the non-equivariant limit of Corollary \ref{cor:twisted}, we see that abelianization holds for these twisted $I$-functions as well. We will denote these non-equivariant, Euler-twisted $I$-functions by $I^{Z\sslash G, E}$.
\end{remark}

\subsection{Big $I$-functions}

The $I$-function we have been discussing in this paper is most directly related to Gromov-Witten invariants with only one insertion, from which one can easily obtain information about invariants with insertions in $H^2(Z\sslash G, \QQ)$ by using the divisor equation. The \textit{big I-function} is related to Gromov-Witten invariants with arbitrary insertions, and it is defined in \cite{bigI} to be the generating series
\begin{equation}\label{eq:bigIfunc}
\II^{Z\sslash G}(z) = 1+ \sum_{\beta\neq 0}q^\beta \II^{Z\sslash G}_\beta(z) \quad \quad \text{where} \quad \quad \II^{Z\sslash G}_\beta(z) = (ev_{\bullet})_*\left(\exp(\evhat_\beta^*(\bt)/z)\frac{[F_\beta]^{vir}}{e_{\CC^*}(N^{vir}_{F_{\beta}})}\right).
\end{equation}
$F_\beta:= F_\beta(Z\sslash G)$ and $ev_\bullet$ are defined as in \eqref{eq:Ifunc}, and the sum is over all $I$-effective classes $\beta$ of $(Z, G, \theta)$ (but we have yet to define the notation $\exp(\evhat_\beta^*(\bt)/z)$). The goal of this section is to prove an abelianization formula for $\II^{Z\sslash G}(z)$ when $Z$ is a vector space, yielding a closed formula for $\II^{Z\sslash G}(z)$ in this situation.

Let $S$ be a torus acting on $Z$ as in Section \ref{sec:equivariant}---we will define \eqref{eq:bigIfunc} $S$-equivariantly. If $Z^s \subset Z$ is any locus of stable points, recall the Kirwan map
\[
\kappa_G: H^*_{S\times G}(Z, \QQ) \rightarrow H^*_S(Z^s/G, \QQ).
\]
We will write out the definition of the Kirwan map when $S$ is trivial. Let $EG$ be the universal principal $G$-bundle. Then we have maps
\[
EG\times_G Z \xleftarrow{a} EG\times_G Z^s \xrightarrow{b} Z^s(G)/G
\]
where $a$ is an open embedding and $b$ is projection to the second factor. Then $b^*$ induces an isomorphism on cohomology, and the Kirwan map is defined by $\kappa_G = (b^*)^{-1}\circ a^*.$ This map is surjective by \cite{Kir84}.

In similar spirit we define
\[
\evhat_\beta^*:H^*_{G\times S}(Z, \QQ)\otimes_{\QQ}\QQ[z] \rightarrow H^*_{S\times\CC^*}(F_\beta, \QQ).
\]
Let $\P\rightarrow F_\beta \times \specialP$ be the universal principal bundle and let $\sigma: F_\beta \times \specialP \rightarrow \P \times_G Z$ be the universal section. When $S$ is trivial, the map $\evhat_\beta^*$ is simply the pullback in cohomology along the composition of maps
\[
F_\beta \xrightarrow{(id, 0)} F_\beta \times \specialP \xrightarrow{\sigma} \P\times_G Z \rightarrow EG \times_G Z.
\]

Now we can define the notation in \eqref{eq:bigIfunc}. Fix a homogeneous basis $\gamma_i$ of $H^*_S(Z\sslash G, \QQ)$. Let $\tilde \gamma_i \in H^*_{S\times G}(Z, \QQ)$ be classes such that $\kappa_G(\tilde\gamma_i) =\gamma_i$, and set
\[
\bt = \sum_i \tilde \gamma_i t_i
\]
for $t_i$ some formal variables. The term $\exp(\evhat_\beta^*(\bt)/z)$ is interpreted as a polynomial in the $t_i$ with coefficients in $H^*_{S\times\CC^*}(F_\beta, \QQ)$ via the power series expansion of the exponential.

When $Z$ is a vector space, we can explicitly compute \eqref{eq:bigIfunc} as follows. By Proposition \ref{prop:chow}, the classes $\gamma_i$ are uniquely determined by their pullbacks $g^*\gamma_i \in H^*_S(Z^s(G)/T, \QQ)$, and these pullbacks may be expressed as $W$-invariant polynomials in the classes $c_1(\L_{\xi_j})$, where $\xi_j$ are the characters of the $T$-action on $Z$. Write 
\[g^*\gamma_i = q_i(c_1(\L_{\bxi}))\]
for these polynomials, where $q_i(c_1(\L_{\bxi}))$ is shorthand for $q_i(c_1(\L_{\xi_1}), \ldots, c_1(\L_{\xi_r}))$.
\begin{corollary}\label{cor:bigI}
The big $I$-function of $Z\sslash_\theta G$ satisfies
\begin{equation}
g^* \II^{Z\sslash G}_{\beta}(z) = j^*\left[ \sum_{\tilde \beta \rightarrow \beta} \exp\left(\sum_it_iq_i(c_1(\L_{\bxi})+\tilde \beta(\bxi)z)/z\right)\left( \prod_{\rho} \frac{\prod_{k=-\infty}^{\tilde\beta(\rho)}(c_1(\L_{\rho}) + kz)}{\prod_{k=-\infty}^0 (c_1(\L_{\rho}) + kz)}\right)I^{Z\sslash T}_{\tilde \beta}(z)\right],
\end{equation}
where 
\[
q_i(c_1(\L_{\bxi})+\tilde \beta(\bxi)z) := q_i(c_1(\L_{\xi_1})+\tilde \beta(\xi_1)z, \ldots, c_1(\L_{\xi_r})+\tilde \beta(\xi_r)z)
\]
and the sum is over all $\tilde \beta$ mapping to $\beta$ under the natural map $\rpic: \Hom(\Pic^T(Z), \ZZ) \rightarrow \Hom(\Pic^G(Z), \ZZ)$ and the product is over all roots $\rho$ of $G$.
\end{corollary}
This corollary extends the procedure in \cite[Sec~5.3]{bigI}.
\begin{proof}
The first step is to carefully choose the lifts $\tilde \gamma_i$. We have a commuting diagram of topological spaces
\[
\begin{tikzcd}
ET\times_T Z\arrow[d, "\psi"] & \arrow[l, "a"'] ET\times_Z Z^s(G) \arrow[r, "b", "\sim"'] \arrow[d] & \arrow[d, "g"] Z^s(G)/T \\
EG \times_G Z & \arrow[l, "a"'] EG\times_G Z^s(G) \arrow[r, "b","\sim"'] & Z^s(G)/G
\end{tikzcd}
\]
which leads to a commuting diagram of cohomology maps
\begin{equation}\label{eq:big1}
\begin{tikzcd}
H_{S\times T}^*(Z, \QQ)^W \arrow[r, "\kappa_T"] & H^*_S(Z^s(G)/T,\QQ)^W\\
H^*_{S\times G}(Z, \QQ) \arrow[u, "\sim"', "\psi^*"] \arrow[r, "\kappa_G"] & H^*_S(Z\sslash G, \QQ) \arrow[u, "\sim", "g^*"']
\end{tikzcd}
\end{equation}
The right vertical arrow is an isomorphism by Proposition \ref{prop:chow}, and the left vertical arrow is an isomorphism by \cite[Prop~1]{Br98}. Let 
\[
\tilde \delta_i = q_i(c_1(L_{\bxi})) \in H^*_{S\times T}(Z, \QQ)^W,
\]
so $\kappa_T(\tilde \delta_i) = g^*\gamma_i$. Then set $\tilde \gamma_i = \psi^*\tilde \delta_i$. Commutativity of \eqref{eq:big1} implies that $\kappa_G(\tilde \gamma_i)=\gamma_i$ as desired.

Now we apply the computation in Section \ref{sec:proof} to $\II_\beta^{Z\sslash G}.$ In place of \eqref{eq:mainproof1} we arrive at the formula
\[
g^*\II_\beta^{Z\sslash G} = \sum_{\tilde \beta \mapsto \beta} \frac{(ev_\bullet)_*\psi_{\tilde \beta}^*(\exp(\evhat_{\tilde \beta}^*(\bt)/z)[F_{\tilde \beta}(Z\sslash G)]^{vir}e_{\CC^*}(N_{F_{\tilde \beta}(Z\sslash G)}^{vir})^{-1}))}{\prod_{\rho \in R^-_{\bdtilde(\tilde\beta)}}c_1(\L_\rho)}
\]
where $\evhat_{\tilde \beta}^*$ denotes the composition
\[
H^*_{G\times S}(Z, \QQ)\otimes_{\QQ}\QQ[z] \xrightarrow{\evhat_\beta^*} H^*_{S\times\CC^*}(F_\beta(Z\sslash G), \QQ) \rightarrow H^*_{S\times\CC^*}(F_{\tilde\beta}(Z\sslash G), \QQ)
\]
where the second map is restriction to an open and closed subspace of $F_\beta(Z\sslash G)$. To compute $\psi_{\tilde \beta}^*\evhat_{\tilde \beta}^*$ we use the commuting diagram
\[
\begin{tikzcd}
H^*_{S\times T}(Z, \QQ) \arrow[r, "\evhat_{\tilde \beta}^*"] & H^*_S(F_{\tilde \beta}(Z\sslash T) \cap Z^s(G), \QQ) \\
H^*_{S\times G}(Z, \QQ) \arrow[r, "\evhat_{\tilde \beta}^*"] \arrow[u, "\psi^*"]& H^*_S(F_{\tilde \beta}(Z\sslash G) , \QQ) \arrow[u, "\psi_{\tilde \beta}^*"]
\end{tikzcd}
\]
which follows from the commuting diagram of topological spaces (in the case when $S$ is trivial)
\[
\begin{tikzcd}
ET\times_T Z \arrow[d, "\psi"] & \T \times_T Z \arrow[l] \arrow[d] & F_{\tilde \beta}(Z\sslash T) \cap Z^s(G) \times \specialP \arrow[l, "\sigma"] \arrow[d]& F_{\tilde \beta}(Z\sslash T) \cap Z^s(G) \arrow[l, "{(id, 0)}"]\arrow[d, "\psi_{\tilde \beta}"]\\
EG\times_G Z  & \P \times_G Z \arrow[l]  & F_{\tilde \beta}(Z\sslash G)  \times \specialP \arrow[l, "\sigma"] & F_{\tilde \beta}(Z\sslash G)  \arrow[l, "{(id, 0)}"]
\end{tikzcd}
\]
We see that
\[
\psi_{\tilde \beta}^*\evhat_{\tilde \beta}^*(\tilde \gamma_i) = \evhat_{\tilde \beta}^*\psi^*(\tilde \gamma_i) = \evhat_{\tilde \beta}^*(\tilde \delta_i)
\]
by the definition of $\tilde \gamma_i$. Finally, we have
\[
\evhat_{\tilde \beta}^*(\tilde \delta_i) = ev_\bullet^*q_i(c_1(\L_{\bxi}) + \tilde \beta(\bxi)z)
\]
This follows from \cite[Lem~5.2,~Rmk~5.3]{bigI} when the GIT chamber of $\theta$ has full dimension, but (as pointed out by the referee) it also follows without any additional hypothesis from our description of the universal family on $F_{\tilde \beta}(Z\sslash T)$ in Proposition \ref{prop:ufam} (for details, see the computation of $n_F^*\TT_\psi$ in the proof of Corollary \ref{cor:pullback}).
The rest of the computation proceeds via the projection formula as in Section \ref{sec:proof} and \cite[Sec~5.1]{bigI}.
\end{proof}

\subsection{Applications to Gromov-Witten theory}
The mirror theorems of \cite{wcgis0} (resp. \cite[Thm~3.3]{bigI}) state that $I^{Z\sslash G}$ (resp. $\II^{Z\sslash G}$) is on the Lagrangian cone when $Z\sslash G$ has a torus action with good properties (but \cite{yang} proves results without a torus action). This means that Theorem \ref{thm:main} and Corollary \ref{cor:bigI} can be translated to statements about $J$-functions, though the translation is simplest when $Z\sslash G$ is sufficiently positive. As an example we will use Theorem \ref{thm:main} to show that the abelianization conjectures \cite[Conj~4.2]{ab-nonab} and \cite[Conj~3.7.1]{frob-ab-nonab} hold in good circumstances.

In order to have a clear statement to use in our application, we summarize \cite[Cor~7.3.2]{wcgis0} here. If $E$ is a vector space with a linear $G$-action, we say that the resulting vector bundle on $Z\sslash G$ is \textit{convex} if, when $(C, \P, \sigma, p_i)$ is a general quasimap to $Z\sslash G$ as in Remark \ref{rmk:general-qmaps}, then $H^1(C, \P\times_G E)=0$. (See \cite[Prop~6.2.3]{stable_qmaps} for some sufficient conditions for $E$ to be convex.)
The following theorem applies to the twisted theory described in Remark \ref{rmk:twist}.
\begin{theorem}[{\cite[Cor~7.3.2]{wcgis0}}]\label{thm:mirror}
Assume $Z\sslash_\theta G$ has an $S$-action with isolated fixed points. Let $E$ be a convex representation satisfying 
\begin{equation}\label{eq:mirror-thm}
\beta(\det(T_Z))-\beta(\det(Z \times E) \geq 2\end{equation}
for all $I$-effective classes $\beta \neq 0$, where $T_Z$ is the ($G$-equivariant) tangent bundle of $Z$. Then $J^{X\sslash G, E} = I^{X\sslash G, E}$, both $S$-equivariantly and nonequivariantly.
\end{theorem}
\begin{proof}
The cited result \cite[Cor~7.3.2]{stable_qmaps} requires us to check \eqref{eq:mirror-thm} for all $\theta$-effective classes $\beta$, whereas we assume that it holds a priori only for $I$-effective classes. However, the proof of \cite[Cor~7.3.2]{stable_qmaps} only uses \eqref{eq:mirror-thm} for $\theta$-effective classes that are realized as the class of some general quasimap (as in Remark \ref{rmk:general-qmaps}) of \textit{genus zero}. Since we have assumed that \eqref{eq:mirror-thm} holds for all $I$-effective classes, it follows from Remark \ref{rmk:effective-generation} that \eqref{eq:mirror-thm} holds for each $\beta$ that is a class of a general genus-zero quasimap .
\end{proof}

Before we can use \cite[Cor~7.3.2]{wcgis0} and Theorem \ref{thm:main} to say something about \cite[Conj~4.2]{ab-nonab} and \cite[Conj~3.7.1]{frob-ab-nonab}, we must address a difference in setup between the current paper and the cited conjectures: our $Z$ is affine, but in \cite{ab-nonab, frob-ab-nonab} $Z$ is projective. The translation is accomplished (at least in many cases) by Lemmas \ref{lem:correspondence} and \ref{lem:correspondence2}.

\begin{lemma}\label{lem:correspondence}
Let $Z$ be a vector space with an action by a reductive group $G$ and character $\theta$ satisfying the assumptions in Section \ref{sec:statement}. Fix a torus $H \simeq \CC^*$. If $\tau: H \hookrightarrow Z(G)$ is a 1-parameter subgroup of the center of $G$ such that the resulting weights of $H$ on $Z$ are all positive and the character $\theta \circ \tau$ also has a positive exponent, then 
\begin{enumerate}
\item $\PP_{\tau}(Z):= (Z \setminus \{0\}) / H$ is a weighted projective space with an action by a reductive group $\overline G = G/H$,
\item $\OO_{\PP_{\tau}(Z)}(\theta) := (Z\times \CC_{\theta\circ \tau})/H$ is a $\overline G$-linearized ample line bundle on $\PP_{\tau}(Z)$,
\item Every semi-stable point for $\OO_{\PP_{\tau}(Z)}(\theta)$ is stable, and
\item $Z\sslash_{\theta} G = \PP_{\tau}(Z)\sslash_{\OO_{\PP_{\tau}(Z)}} \overline G$
\end{enumerate}
\end{lemma}
\begin{proof}
Statement (1) is an immediate consequence of the hypotheses. The $\overline G$-action in statement (2) comes from the diagonal $G$-action on $V \times \CC_{\theta}$ (here, it is important that $\tau$ lands in the center of $G$). For $n \in \ZZ_{> 0}$, we have a fiber square
\[
\begin{tikzcd}
(Z \setminus \{0\})\times\CC_{n\theta} \arrow[r] \arrow[d] & \OO_{\PP_{\tau}(Z)}(n\theta) \arrow[d] \\
Z\setminus \{0\} \arrow[r, "\pi"] & \PP_\tau(Z)
\end{tikzcd}
\]
where the horizontal maps are $H$-torsors and are also equivariant with respect to the projection homomorphism $G \rightarrow \overline G.$ Hence pullback defines a 1-to-1 correspondence between $G$-equivariant sections of $(Z\setminus \{0\})\times \CC_{n\theta} \rightarrow Z\setminus \{0\}$ and $\overline G$-equivariant sections of $\OO_{\PP_\tau(Z)}(n\theta) \rightarrow \PP_\tau(Z)$. The zero locus of a section of $(Z\setminus \{0\})\times \CC_{n\theta}$ is the inverse image under $\pi$ of the zero locus of the corresponding section of $\OO_{\PP_\tau(Z)}(n\theta) \rightarrow \PP_\tau(Z)$. This shows (3). Finally, (4) follows from descent for closed subschemes.
\end{proof}

\begin{example}\label{ex:standard}
Lemma \ref{lem:correspondence} applies if $Z$ is a vector space with a \textit{linear} $G$-action that contains the dilations, and if the composition of $\theta$ with a dilation yields a character $\CC^*\rightarrow \CC^*$ with a positive exponent. 
\end{example}

As pointed out by the referee, Lemma \ref{lem:correspondence} always applies in the following situation.

\begin{lemma}\label{lem:correspondence2}
Let $(Z, G, \theta)$ be a triple satisfying the assumptions in Section \ref{sec:statement} with $Z$ a vector space. If $Z\sslash T$ is projective, then the hypotheses of Lemma \ref{lem:correspondence} are satisfied.
\end{lemma}
\begin{proof}
Let $\Sigma \subset \chi(T)$ be the cone in the character lattice of $T$ generated by the weights of $V$ with respect to $T$. Since these weights come from a $G$-representation, $\Sigma$ is sent to itself by the action of $W$ on $\chi(T)$. Since $Z\sslash T$ is projective, by \cite[Prop~14.3.10]{CLS}, $\Sigma$ is strongly convex, hence by \cite[Prop~1.1.12]{CLS} the dual $\Sigma^\vee$ in the lattice of cocharacters has nonempty interior. In other words, we have a 1-parameter subgroup $\tau': H \rightarrow T$ such that the resulting weights of $H$ on $Z$ are all positive. 

Now set $\tau = \sum_{w \in W} w\cdot \tau'$. Since $W$ sends $\Sigma$ and hence $\Sigma^\vee$ to itself, $\tau$ is also in the interior of $\Sigma^\vee$. By definition of $\tau$ we have $w\cdot \tau = \tau$; plugging in $h\in H$ we see that $w\cdot \tau(h) = \tau(h)$, so that $\tau$ is actually a 1-parameter subgroup of (the identity component of) $T^W.$ Since the inclusion $Z(G) \hookrightarrow T^W$ is an isomorphism on identity components (see e.g. \cite{overflow}), we see that $\tau$ is central. 

Finally, since $Z\sslash_{\theta} T$ is projective, $\theta$ is not trivial. Since $Z^{ss}_{\theta}(T)$ is not empty we see that for some $k \geq 0$ the character $k\theta$ is a nonnegative linear combination of the weights, and some coefficient is strictly positive. Hence $\theta \circ \tau$ also has positive exponent. 
\end{proof}

Now we are ready to study the conjecture \cite[Conj~3.7.1]{frob-ab-nonab}, which is a priori a statement about the Frobenius manifolds defined by the Gromov-Witten theories of $Z\sslash T$ and $Z\sslash G$.

\begin{corollary}\label{cor:fm-abelianization}
Let $(Z, G)$ be a pair satisfying the hypotheses of Lemma \ref{lem:correspondence} and
let $S$ be a torus acting as in Section \ref{sec:equivariant}. Assume that $Z^{us}$ has codimension at least 2. If $Z\sslash G$ is Fano of index at least 2, and if the $S$-action has isolated fixed points, then the conjecture \cite[Conj~3.7.1]{frob-ab-nonab} holds.
\end{corollary}

We make the (mild) assumption on codimension in Corollary \ref{cor:fm-abelianization} to be in agreement with the context of the cited conjecture. 

\begin{proof}
Our strategy is to first apply the reconstruction result \cite[Thm~4.3.6]{frob-ab-nonab} to reduce the Frobenius manifold correspondence to the relationship of ``small'' $J$-functions in \cite[Conj~4.2]{ab-nonab}. The latter statement essentially Theorem \ref{thm:main} combined with \cite[Cor~7.3.2]{wcgis0}, but in executing this strategy we have to be careful in a few places. 

We first resolve the above-mentioned difference in setup: our $Z$ is affine but in \cite{ab-nonab, frob-ab-nonab} $Z$ is projective. By Lemma \ref{lem:correspondence}, the quotient $Z\sslash_{\theta} G$ is identified with the projective quotient $\PP_\tau(Z)\sslash_{\OO_{\PP_{\tau}(Z}}\overline G$. Since $\tau: H \rightarrow T \subset G$ is injective, it is also split, so $\overline T = T/H$ is a torus. We use this as a maximal torus of $\overline G$. 

Now the reconstruction theorem in \cite[Thm~4.3.6]{frob-ab-nonab} applies because the localized equivariant cohomology ring
\[
H^*_S(Z\sslash G, \QQ) \otimes \mathrm{Frac}(H^*_S(pt, \QQ))
\]
is generated by divisors (this follows from the torus localization theorem). Also, since $Z\sslash G = \PP_\tau(Z)\sslash_{\OO_{\PP_{\tau}(Z}}\overline G$ is Fano of index at least 2, the inequality \eqref{eq:mirror-thm} holds (with $E=0$). 

We must be careful one more time: the ``small $J$-function'' in \cite{ab-nonab, frob-ab-nonab} differs from the small $I$-function considered in this paper; in fact, it corresponds to the \textit{big} $I$-function in \eqref{eq:bigIfunc} with the classes $\gamma_i$ in $\bt$ restricted to a basis of $H^2(Z\sslash G, \QQ)$. However, Theorems \ref{thm:main} and \ref{thm:mirror} yield the expected formulas for these ``middle-sized'' $I$ and $J$-functions as explained in \cite[404]{wcgis0}. Now the Corollary follows from Theorem \ref{thm:main} and the mirror theorem of \cite[Cor~7.3.2]{wcgis0} (restated in Theorem \ref{thm:mirror} above).
\end{proof}

\subsection{Example: A Grassmannian bundle on a Grassmannian variety}\label{sec:example}
To illustrate the geometry in Proposition \ref{prop:diagram} and to give an example of applying Theorem \ref{thm:main}, in this section we investigate a family of Fano hypersurfaces.

\begin{theorem}\label{thm:example}
Let $Q:=Gr_{Gr(k,n)}(\ell,U^{\oplus m})$ be the Grassmannian bundle of $\ell$-planes in $m$ copies of the tautological bundle $U$ on $Gr(k, n)$, and assume $n - \ell m \geq 2$ and $km \geq 3$. Let $\D$ be {the dual of} the determinant of the tautological bundle on $Q$. Let $S=(\CC^*)^{n+m}$ act on $Q$ and $\D$ as defined in Section \ref{sec:equivariant_example}. Then the $\D$-twisted equivariant small $J$-function of $Q$ equals $1+\sum_{d,e>0} q_1^dq_2^eJ_{(d,e)}(z)$, where $J_{(d,e)}(z)$ equals
\begin{equation}\label{eq:example}
\begin{aligned}
  \sum_{\substack{d_1+\ldots+d_k=d\\e_1+\ldots+e_\ell=e}}
 &\prod_{\substack{i,j=1\\i\neq j}}^k\left( \frac{\prod_{h=-\infty}^{d_i-d_j}(x_i-x_j + hz)}{\prod_{h=-\infty}^{0}(x_i-x_j + hz)}\right)
 \prod_{\substack{i,j=1\\i\neq j}}^\ell\left( \frac{\prod_{h=-\infty}^{e_i-e_j}(y_i-y_j + hz)}{\prod_{h=-\infty}^{0}(y_i-y_j + hz)}\right)\\
 \mathbf{\cdot}&
 \prod_{i=1}^k\prod_{\alpha=1}^n\left( \frac{\prod_{h=-\infty}^0(x_i+\lambda^1_\alpha+hz)}{\prod_{h=-\infty}^{d_i}(x_i+\lambda^1_\alpha+hz)}\right)\prod_{i=1}^k\prod_{j=1}^\ell\prod_{\beta=1}^m\left(\frac{\prod_{h=-\infty}^0(y_j-x_i+\lambda_\beta^2+hz)}{\prod_{h=-\infty}^{e_j-d_i}(y_j-x_i+\lambda_\beta^2+hz)} \right)\\
 \mathbf{\cdot}&\left(\frac{\prod_{h=-\infty}^{e}(\sum_{j=1}^{\ell} y_j+ hz)}
 {\prod_{h=-\infty}^0 (\sum_{j=1}^{\ell} y_j+ hz)}\right),
\end{aligned}
\end{equation}
where the $x_i$ are the Chern roots of the dual of $U$, the $y_j$ are the Chern roots of the dual of the tautological bundle on $Q$, and the $\lambda$'s are the equivariant parameters. 
\end{theorem}
In the formula \eqref{eq:example}, the first line is the factor coming from the roots of $G$, the second line is the $J$-function of the abelian quotient, and the last line is the $\D$-twisting factor.

\subsubsection{Defining the target}
To define the GIT target, choose integers $k, n, \ell,$ and $m$ with $k<n$ and $\ell < km$. Let $M_{k\times n}$ denote the space of $k\times n$ matrices with complex entries, and set
\begin{itemize}
\item the vector space $Z = M_{k\times n} \times M_{\ell \times km}$
\item the group $G = GL_k \times GL_\ell$
\item the action $(g, h)\cdot(X, Y) = (gX, h Y \mathrm{diag}(g^{-1}))$ for $(g, h) \in G$ and $(X, Y)\in X$ where $\mathrm{diag}(g^{-1})$ is the block diagonal $km\times km$ matrix with $g^{-1}$ repeated $m$ times
\item the character $\theta(g, h) = \det(g)\det(h)$
\end{itemize}
For the maximal torus $T \subset G$ choose the group of diagonal matrices.
We can check stability of points using the numerical criterion \cite[Prop~2.5]{king}. It is straightforward to compute that 
\[Z^{ss}_{\theta}(G) = Z^s_{\theta}(G) = (M_{k\times n}\setminus \Delta) \times (M_{\ell\times km} \setminus \Delta),\]
where $\Delta$ denotes matrices of less than full rank. 
Thus, 
\[Z\sslash_{\theta}G = Gr_{Gr(k, n)} (\ell, U^{\oplus m}) =:Q
\]is the Grassmannian bundle of $\ell$-planes  in $m$ copies of the tautological bundle $U$ on $Gr(k, n)$.
Observe that $\det(T_Z)$ is the $G$-equivariant line bundle on $Z$ corresponding to the character
\[
(g,h)\mapsto \det(g)^{n-\ell m}\det(h)^{km}.
\]
The line bundle $\D$ is the $G$-equivariant line bundle on $Z$ corresponding to the character
\begin{equation}\label{eq:D-char}
(g, h) \mapsto \det(h).
\end{equation}

\subsubsection{Quasimaps and I-function}
To make \eqref{eq:main} explicit for our chosen target we must compute the $I$-effective classes (Definition \ref{def:effective}). As an illustration of Proposition \ref{prop:diagram} we will also describe $F_{\tilde \beta}(Z\sslash G)$. A stable quasimap to $Q$ is equivalent to the following data:
\begin{itemize}
\item a rank-$k$ vector bundle $\oplus_{i=1}^k \OO_{\PP^1}(d_i)$ and a rank-$\ell$ vector bundle $\oplus_{j=1}^\ell \OO_{\PP^1}(e_j)$
\item a section $\sigma$ of $\left[\oplus_{i=1}^k \OO_{\PP^1}(d_i)^{\oplus n}\right] \oplus \left[\oplus_{j=1}^\ell \oplus_{i=1}^k \OO_{\PP^1}(e_j-d_i)^{\oplus m}\right],$ written as a $k\times n$ and $\ell\times mk$ matrix of polynomials, such that all but finitely many points $\bx \in \PP^1$ satisfy $\sigma(\bx) \in Z^s.$
\end{itemize}
This data defines a quasimap to $Z\sslash T$ of class $\tilde \beta = (d_1, \ldots, d_k, e_1, \ldots, e_\ell) \in \Hom(\chi(T), \ZZ)$; the class as a quasimap to $Z\sslash G$ is $\beta=(\sum d_i, \sum e_j)\in \Hom(\chi(G), \ZZ).$ In order to have finitely many basepoints, a stable quasimap must have $d_i \geq 0$, hence also $e_j\geq 0$. Now we can check that \eqref{eq:mirror-thm} is satisfied: if $\beta=(d,e)$ is $I$-effective, then
\[
\beta(\det(T_Q)) - \beta(\det(\D)) = (n-\ell m)d + (km-1)e \geq 2d+2e \geq 2
\]
when $(d, e) \neq (0,0)$ (the first inequality uses our assumptions in Theorem \ref{thm:example}).

Finally we describe $F_{\tilde \beta}(Z\sslash G)$. For simplicity assume that the sequences $d_1, \ldots, d_k$ and $e_1, \ldots, e_\ell$ are ordered from smallest to largest. The subspace $Z_{\tilde \beta} \subset X$ is $M_{k\times n} \times Z_{\tilde \beta}'$, where $Z_{\tilde \beta}'$ is the subspace of $M_{\ell\times km}$ consisting of matrices $(m_{ij})$ where 
\[m_{ij}=0 \quad\quad \text{if} \quad \quad e_i-d_{(j\;\mathrm{mod}\; m)+1}<0.\]
Such a matrix looks something like
\[
\begin{tikzpicture}[xscale=.3, yscale=-.3]
\draw (-.2,4.5)--(-.5,4.5)--(-.5,-.5)--(-.2,-.5);
\draw (0,0)--(1.8,0)--(1.8,1)--(2.8,1)--(2.8,3)--(4.8,3)--(4.8,4)--(0,4)--(0,0);
\draw (2,0)--(2,.8)--(3,.8)--(3,2.8)--(5,2.8)--(5,4)--(6,4)--(6,0)--(2,0);
\node at (1.5,2.2) {*};
\node at (4.3,1.7) {0};
\begin{scope}[shift={(6.2,0)}]
\draw (0,0)--(1.8,0)--(1.8,1)--(2.8,1)--(2.8,3)--(4.8,3)--(4.8,4)--(0,4)--(0,0);
\draw (2,0)--(2,.8)--(3,.8)--(3,2.8)--(5,2.8)--(5,4)--(6,4)--(6,0)--(2,0);
\node at (1.5,2.2) {*};
\node at (4.3,1.7) {0};
\node at (7,2) {$\cdot$};
\node at (8,2) {$\cdot$};
\node at (9,2) {$\cdot$};
\end{scope}
\begin{scope}[shift={(16,0)}]
\draw (0,0)--(1.8,0)--(1.8,1)--(2.8,1)--(2.8,3)--(4.8,3)--(4.8,4)--(0,4)--(0,0);
\draw (2,0)--(2,.8)--(3,.8)--(3,2.8)--(5,2.8)--(5,4)--(6,4)--(6,0)--(2,0);
\node at (1.5,2.2) {*};
\node at (4.3,1.7) {0};
\draw (6.2,4.5)--(6.5,4.5)--(6.5,-.5)--(6.2,-.5);
\end{scope}
\end{tikzpicture}
\]
where the entries labeled ``0'' are required to be zero and the entries labeled ``*'' are not. The same $\ell \times k$ pattern of *'s and 0's is repeated $m$ times.
The group $P_{\tilde \alpha}\subset G$ is $P_1\times P_2$ where $P_1$ is the parabolic subgroup of $GL_k$ equal to block lower triangular matrices with blocks determined by the multiplicities of the $d_i$, and $P_2 \subset GL_\ell$ is similarly defined by the $e_j$. Hence in Proposition \ref{prop:diagram} we have a series of maps
\[
F_{\tilde \beta}(Z\sslash G) = Z^s_{\tilde \beta}(G)/P_{\tilde \alpha} \hookrightarrow Z^s(G)/P_{\tilde \alpha} \rightarrow Z^s(G)/G = Z\sslash G
\]
whose composition is $ev_\bullet$. The first arrow is a closed embedding and the second is a flag bundle.

\subsubsection{A good torus action}\label{sec:equivariant_example}
The target $Q$ has a torus action with isolated fixed points. Let $S = (\CC^*)^n \times (\CC^*)^m$ act on $Z$ as follows: if $s_1$ is an $n\times n$ diagonal matrix and $s_2$ is a $km\times km$ diagonal matrix with $m$ constant $k\times k$ diagonal blocks, and both $s_1$ and $s_2$ are filled with numbers from $\CC^*$ then
\[
(s_1, s_2)\cdot (X,Y) = (Xs_1, Ys_2).
\]
This action commutes with the action of $G$. We can extend it to a linearization of $\D$ as follows. The total space of $\D$ is $X \times_G \CC$, where $G$ acts on $\CC$ via the character \eqref{eq:D-char}. Define $(s_1, s_2)\cdot(X, Y,z) = (Xs_1, Ys_2, z)$ for $(X,Y)\in Z$ and $z\in\CC$.

The $S$-action on $Q$ has isolated fixed points as follows.
For $I\subset\{1, \ldots, n\}$ let 
$D_I$ denote the $k\times n$ matrix which has the identity matrix in the $I$-columns and zeros elsewhere. 
Similarly, for $J\subset\{1, \ldots, km\}$, let $D_J$ denote the 
$\ell\times km$ matrix which has the identity matrix in the $I$-columns and zeros elsewhere. 
Then the fixed points of the $S$-action are $(D_I, D_J)$ for all possible combinations of $I$ and $J$. 

Now we can read off the twisted $I$-function using Corollary \ref{cor:twisted}. We apply the mirror theorem in \cite[Cor~7.3.2]{wcgis0} (stated here as Theorem \ref{thm:mirror}) to conclude that \eqref{eq:example} holds.

\printbibliography

\end{document}